\documentclass[11pt]{amsart}

          \usepackage{amssymb}
          \usepackage{amsmath}
          \usepackage{amsthm}
          \usepackage{amsfonts}
	  \usepackage{mathabx}
          \usepackage[english]{babel}
	  \usepackage[T1]{fontenc}
          \usepackage[utf8]{inputenc}
          \usepackage{enumerate}
          \usepackage{graphicx}
          \usepackage{url}
          \usepackage{color}
	   \usepackage{a4wide}
	  \usepackage[numbers,sort&compress]{natbib}
\usepackage{float}

\def\cal{\mathcal}
\newcommand{\comment}[1]{}
\newcommand{\ind}{{\bf 1}}
\def\indd#1{{\bf 1}_{\{#1\}}}

\newcommand{\proba}{\mathbb P}
\newcommand{\esp}{{\mathbb E}}

\newcommand{\sign}{{\rm{sign}}}
\newcommand{\defe}{\mathrel{\mathop:}=}
\newcommand{\inv}{^{-1}}

\newcommand{\calA}{{\cal A}}
\newcommand{\calB}{{\cal B}}
\newcommand{\calC}{{\cal C}}

\newcommand{\calL}{{\cal L}}
\newcommand{\calM}{{\cal M}}
\newcommand{\calN}{{\cal N}}

\newcommand{\calS}{{\cal S}}

\def\C{{\mathbb C}}

\newcommand{\eqnh}{\begin{eqnarray*}}
\newcommand{\eqne}{\end{eqnarray*}}
\newcommand{\eqnhn}{\begin{eqnarray}}
\newcommand{\eqnen}{\end{eqnarray}}
\newcommand{\equh}{\begin{equation}}
\newcommand{\eque}{\end{equation}}

\def\summ#1#2#3{\sum_{#1 = #2}^{#3}}

\def\sif#1#2{\sum_{#1=#2}^\infty}

\newcommand{\eqd}{\stackrel{\rm d}{=}}

\def\topp#1{^{(#1)}}

\def\nn#1{{\left\|#1\right\|}}

\def\snn#1{\|#1\|}

\def\abs#1{\left|#1\right|}

\def\ccbb#1{\left\{#1\right\}}

\def\pp#1{\left(#1\right)} 

\def\d{{\rm d}}

\def\P{{\mathbb P}}

\def\mand{\mbox{ and }}

\def\qmwith{\quad\mbox{ with }\quad}
\def\mfa{\mbox{ for all }}

\def\mmas{\mbox{ as }}

\def\what#1{\widehat{#1}}

\def\weakto{\Rightarrow}

\def\R{{\mathbb R}}
\def\Rd{{\mathbb R^d}}

\def\N{{\mathbb N}}

\def\RdRp{{\Rd\times\R_+}}
\def\S{\cS}
\def\cal#1{\mathcal{#1}}
\newcommand{\cL}{{\mathcal L}}
\newcommand{\cS}{{\mathcal S}}
\newcommand{\Leb}{{\rm Leb}}

\newtheorem{Thm}{Theorem}[section]
\newtheorem{Lem}[Thm]{Lemma}
\newtheorem{Prop}[Thm]{Proposition}

\theoremstyle{definition}
\newtheorem{Rem}[Thm]{Remark}

\numberwithin{equation}{section}

\begin{document}

\title[Generalized operator-scaling random ball model]{Generalized operator-scaling random ball model}

\author{Hermine Bierm\'e}
\address
{
Hermine Bierm\'e,
Laboratoire de Math\'ematiques et Applications UMR CNRS 7348,
Universit\'e de Poitiers, 
Boulevard Marie et Pierre Curie
86962 Futuroscope Chasseneuil Cedex, France.
}
\email{hermine.bierme@math.univ-poitiers.fr}

\author{Olivier Durieu}
\address{
Olivier Durieu\\
Institut Denis Poisson, UMR-CNRS 7013\\
Universit\'e de Tours, Parc de Grandmont, 37200 Tours, France.
}
\email{olivier.durieu@univ-tours.fr}

\author{Yizao Wang}
\address
{
Yizao Wang,
Department of Mathematical Sciences,
University of Cincinnati,
2815 Commons Way, ML--0025,
Cincinnati, OH, 45221-0025.
}
\email{yizao.wang@uc.edu}

\date{\today}

\begin{abstract}
This article introduces the operator-scaling random ball model, generalizing the isotropic random ball 
models investigated recently in the literature to anisotropic setup. The model is introduced as a generalized random field and results on weak convergence are established in the space of tempered distributions.
\end{abstract}


\maketitle

\section{Introduction}

In the past ten years, random ball models have appeared as a simple and yet flexible  class of random fields that characterize various types of spatial dependence structures \citep{kaj07scaling,breton09rescaled,bierme10selfsimilar,breton11functional,gorgens14gaussian,gobard15random,breton15infinite,bierme06poisson,pili2016}. In particular, in several regimes, their scaling limits are self-similar and with long-range dependence 
\citep{samorodnitsky16stochastic,pipiras17long,beran13long}. 
Such properties are desirable when modeling various real world phenomena and thus such results have a broad range of applications.

In words, a random ball model consists in  a collection of random balls in $\Rd$ with locations following a homogeneous Poisson point process and with independent and identically distributed random radius and weights.
Thus, each realization of random balls on the space can be naturally viewed as a linear functional on an appropriate space of test functions. Asymptotic behaviors are then of interest, when all the balls are simultaneously rescaled by a parameter $\rho$, and at the same time the intensity of balls also changes with respect to $\rho$.  Under mild assumption on the distribution of the radius, limit theorems can be established for $\rho\to 0$ or $\rho\to \infty$, corresponding to the zoom-out or zoom-in cases respectively. 
In both cases, the qualitative behavior of the limit random fields, whether exhibiting spatial dependence or not, depends on whether the random balls are dense or sparse in the limit, in certain sense to be specified below.

The random ball models can be viewed as generalizations of certain one-dimensional models based on Poisson point processes that appeared in the study of Internet traffics, see for example \citep{mikosch02network,kaj08convergence} and references therein.
However, the extension to  high dimensions presents
 new technical challenges, and should not be viewed as simple generalization of the one-dimensional results. In particular, the developments  until now
  have two main limitations. 
First, results so far in the literature focus on isotropic random ball models (except for \cite{pili2016}). That is, the random fields have the same distribution in each different direction. This feature, from the application point of view, makes the model much less attractive.  
Second,  the tightness of the scaled random fields is difficult to establish. Usually random ball models are defined as a random field $\{X(\mu)\}_{\mu\in\calM}$ indexed by a family of measures $\calM$ on $\Rd$. The tightness of such random fields, after appropriate normalizations, is only established for very restricted classes of $\calM$ \citep{breton15infinite,breton11functional}. 

The goal of this paper is to establish limit theorems for a general class of random ball models, and to remove the aforementioned two limitations. 

First, we provide a general framework of random ball models exhibiting anisotropic features and hence include all previously considered ones as special cases.
It is now well understood that a natural generalization of notion of self-similarity, widely used in stochastic processes and time series, is the so-called {\em operator-scaling} property for random fields introduced in \citet{bierme07operator}. A random field $\{Z_t\}_{t\in\Rd}$ is said to be $(E,H)$-operator-scaling, if
\begin{equation}\label{def:op-sc}
\ccbb{Z_{c^E t}}_{t\in\Rd}\eqd c^H\ccbb{Z_t}_{t\in\Rd}, \mfa c>0,
\end{equation}
where $E$ is an appropriate $d\times d$ matrix, $c^E:= \sif k0(E\log c)^k/k!$ is also a matrix, and $H>0$.
Taking $E$ to be the identity matrix, the above says that the random field $Z$ is self-similar. The motivation of allowing general matrix $E$ is to generalize this notion to anisotropic random fields. Such random fields are often of practical importance in various applications, and they also present theoretical challenges. Families of anisotropic random fields are known, and path properties have been investigated. See for example \citep{bierme09holder,li15exact,meerschaert13,xiao09sample}.
At the same time, the development of limit theorems for anisotropic random fields is still at an early stage. For some recent results, see for example 
\citep{bierme17invariance,li12occupation,wang14invariance,puplinskaite15scaling,lavancier07invariance,durieu17random,shen17operator}.
In this article, we also consider more general random sets than balls, precisely sets of finite perimeter.

Second, we view the random ball models as {\em distribution-valued random elements}, also known as {\em generalized random fields}, and establish weak convergence in the space of tempered distributions. 
 A complete description of self-similar generalized Gaussian random fields was obtained in \cite{dobrushin79gaussian} and allows to obtain essentially all Gaussian, translation- and rotation-invariant, $H$-self-similar generalized random field as scaling limits
of a random balls model in \cite{bierme10selfsimilar}. Beyond the Gaussian framework, generalized L\'evy random field, including stable generalized random field 	have been investigated in \cite{UnserBook}, where they are named as sparse stochastic processes. 
Distribution-valued random variables and stochastic processes are already widely used to describe fluctuations of empirical measures of complex particle systems, including notably interacting particle systems \citep{kipnis99scaling} and branching particle systems~\citep{holley78generalized,kipnis99scaling,bojdecki07long,li12occupation}, just to mention a few.

The paper is organized as follows. Section~\ref{sec:prelim} presents background on generalized random fields, the precise definition of the random ball model, and the four regimes of convergence that we investigate. 
The limit theorems are stated in Section~\ref{sec:limit}, while their proofs are postponed in Section~\ref{sec:proof}.
In Section~\ref{sec:prop}, we study statistical properties of the limit random fields. To conclude, a pointwise representation is obtained in Section~\ref{sec:represent} and some illustrations are given in the appendix.

Throughout, $C$ stands for real constants that may change values from line to line. Without ambiguity, for $x\in\Rd$, $|x|$ denotes its Euclidean norm. We write $a\vee b = \max(a,b)$ and $a\wedge b = \min(a,b)$ for $a,b\in\R$.


\section{Background and definitions}\label{sec:prelim}

\subsection{Generalized random fields}
The standard references for generalized random fields include notably \citep{gelfand64generalized,gelfand64generalizedV1,dobrushin79gaussian,kallianpur95stochastic,fernique67processus}.
In words, these fields are defined as random variables with values in a space of distributions (or generalized functions).
To this end we consider the Schwartz space $\cS(\R^d)$ of all
real-valued infinitely differentiable rapidly decreasing functions on $\R^d$, and $\cS'(\R^d)$
its topological dual, the space of tempered distribution.  As usual $\cS(\R^d)$ is equipped
with the topology that corresponds to the following notion of
convergence: $f_n\rightarrow f$ if and only if for all
$N\in\N := \{0,1,2,\dots\}$ and $j=(j_1,\ldots,j_d)\in\N^d$
\[
\nn {f_n-f}_{N,j} := \sup_{z \in \R^d}(1 + | z |)^N \left|
D^{j}\left({f_n}-f\right)(z)\right|\rightarrow 0,\text{ as }n\to\infty,
\] 
where $D^jf(z)=\frac{\partial^{j_1}\cdots\partial^{j_d}}{\partial
z_1^{j_1}\cdots \partial z_d^{j_d}}f(z)$ denotes the partial
derivative of order $j$.

We will actually also consider the space
\[
\cS_1(\R^d):=\left\{f \in \cS(\R^d);\; \int_{\R^d}f(z)dz=0\right\}.
\]
Note that $\cS_1(\R^d)=\mbox{span}\left\{ D^{j}f ;\; f \in \cS(\R^d),j\in \{0,1\}^d, j_1+\cdots+j_d=1\right\}$.
For convenience, we also write $\cS_0(\R^d)=\cS(\R^d)$ and thus we will be able to use $\cS_n(\R^d)$ for $n\in\{0,1\}$ in the sequel.
We denote by $\cS_n'(\R^d)$ the topological dual of $\cS_n(\R^d)$ and by $(\,\cdot,\cdot\,)$ the duality bracket. We usually consider two distinct topologies on $\cS_n'(\R^d)$.
The strong topology is induced by the family of semi-norms
\[
 q_B(\cdot)=\sup_{f\in B}|(\,\cdot\,,f)|,\quad B \mbox{ bounded in }\cS_n(\R^d).
\]
The weak topology on $\cS_n'(\R^d)$ is the topology induced by the family of semi-norms $|(\,\cdot\,,f)|$,  $f \in \cS_n(\R^d)$.
A first remark is that both topologies generate the same Borel $\sigma$-field denoted by $\calB(\calS_n'(\Rd))$, see \cite{LT}.

A generalized random field is an $\cS_n'(\R^d)$-valued random variable, that is a measurable mapping $X$ from a probability space $({\Omega},{{\mathcal A}},\P)$ to $(\cS_n'(\R^d),\calB(\calS_n'(\Rd)))$. 
For such a generalized random field $X$, we let its evaluation at $f\in \calS_n(\R^d)$ be denoted by $X(f)$, which is a real random variable on the same probability space.

The law of a generalized random field $X$ is uniquely determined by its characteristic functional 
\[
\cL_{X}(f) := \int_{\Omega}e^{i X(f)}\,d\proba, \quad f\in\calS_n(\Rd).
\]
Further, $X$ induces a family of random variables $X(f)$ on $(\Omega,{\mathcal A})$ indexed by $f\in \calS_n(\Rd)$, with characteristic functions given by
$$
{\mathbb{E}}\left(e^{it X(f)}\right)=\int_{\Omega}e^{itX(f)}d{\mathbb{P}}=\cL_{X}(tf),\quad t\in\R.
$$
By linearity, the finite-dimensional distributions of $X$ are simply obtained
with
\[
\cL_{X}(a_1f_1+\cdots+a_kf_k)={\mathbb{E}}\left(e^{i[a_1{X}(f_1)+\cdots+a_k{X}(f_k)]}\right),
\]
for all $k\ge 1$, $a_1,\ldots,a_k\in\R$ and $f_1,\ldots,f_k\in \cS_n(\R^d)$.

In practice, however, given a family of real random variables $\{X(f)\}_{f\in\calS_n(\Rd)}$ on a probability space $(\Omega,\calA,\proba)$ satisfying
\equh\label{eq:linear}
X(af+bg) = aX(f)+bX(g) \mbox{ a.s. } \mfa a,b\in\R, f,g\in\cS_n(\Rd),
\eque
a priori it is not clear whether a corresponding $\calS_n'(\Rd)$-valued random variable exists. When this can be achieved, namely if there exists an 
$\cS_n'(\R^d)$-valued random variable $\tilde{X}$, possibly defined on another probability space $(\tilde{\Omega},\tilde{\mathcal A},\tilde{\mathbb{P}})$, such that
 for all $k\ge 1$, $f_1,\ldots,f_k\in \cS_n(\R^d)$,
$A_1,\ldots,A_k \in {\mathcal B}(\R)$,
\[
\mathbb{P}(X(f_1)\in A_1,\ldots,X(f_k)\in A_k)=\tilde{\mathbb{P}}\pp{\tilde{X}(f_1)\in A_1,\ldots,\tilde{X}(f_k)\in A_k},
\]
we say that $\tilde X$ is a version of $X=\{X(f)\}_{f\in \cS_n(\R^d)}$
\citep[Definition 9.1.1]{samorodnitsky94stable}. 
Let us quote that this notion is weaker than the notion of regularization in \citep{ito83distribution}. Actually, a regularization $\tilde{X}$ of $X$ should be defined on the same probability space $(\Omega, {\mathcal A},\mathbb{P})$ than $X$ and satisfies $\tilde{X}(f)=X(f)$ a.s. for all $f\in \cS_n(\R^d)$. 
However, when we deal with convergence in law for most of the part of the paper, the notion of version is enough for our purpose: once the existence of a version is proved, it suffices to work with the characteristic functionals of the original individual random variables. At only a few occasions we shall establish results in the stronger notion of regularization.

We recall below two fundamental theorems when working with limit theorems of generalized random fields, both based on characteristic functionals. 
The following theorem is a direct consequence of Minlos--Bochner's theorem, see \cite[Corollary 2.2]{LT}.
\begin{Thm}\label{thm:version}
Let $X=\{X(f)\}_{f\in \cS_n(\R^d)}$ be a collection of real random variables on $(\Omega,{\mathcal A},\mathbb{P})$ satisfying \eqref{eq:linear}.
If ${\mathcal L}_X: \cS_n(\R^d) \rightarrow \C$ is continuous then $X$ admits a version that is an
$\cS_n'(\Rd)$-valued random variable.
\end{Thm}

Recall that a sequence of generalized random fields $\{X_m\}_{m\ge1}$ converges in distribution to $X$, denoted by $X_m\weakto X$, in $\cS'_n(\Rd)$ given the strong topology if  for all $\varphi:\cS_n'(\R^d)\to\R$ continuous for the strong topology and bounded,
$$
\int_{\cS_n'(\R^d)}\varphi(u)d\mathbb{P}_{X_m}(u)\underset{m\rightarrow
\infty}{\longrightarrow} \int_{\cS_n'(\R^d)}\varphi(u)d\mathbb{P}_{X}(u).
$$
Similarly, $X_m\weakto X$ in $\cS'_n(\Rd)$ given the weak topology, if the above holds for all $\varphi:\cS_n'(\Rd)\to\R$ that is bounded and continuous with respect to the weak topology.
 As a consequence of L\'evy's continuity theorem (\cite[Theorem 2.3]{LT}), we can state the following result, see \cite[Corollary 2.4]{LT}.
\begin{Thm}\label{thm:levy}
Let $\{X_m\}_{m\ge1}$, $X$ be 
$\cS_n'(\Rd)$-valued random variables. The following conditions are equivalent:
\begin{itemize}
\item $X_m\weakto X$ in $\cS_n'(\R^d)$ given the strong topology,
\item $X_m\weakto X$ in $\cS_n'(\R^d)$ given the weak topology,
\item $\cL_{X_m}(f)\to \cL_X(f)$ for all $f\in\cS_n(\R^d)$.
\end{itemize}
\end{Thm}
Since both notions of convergence are equivalent, we shall just write $X_m\weakto X$ in $\cS_n'(\R^d)$ in the sequel.
\begin{proof}[Proofs of Theorems \ref{thm:version} and \ref{thm:levy}]
We refer to \cite{fernique67processus} for the stated results in the more general framework in terms of nuclear spaces. For the special case $\cS'(\Rd) \equiv \cS'_0(\Rd)$, we refer  to \cite{LT} where self-contained and simplified proofs can be found. Results in \cite{LT} can then be extended for $\cS_1'(\R^d)$ by the following idea
from \citet[Proposition 2.1]{dobrushin79gaussian}. Let us quote that fixing a function $\psi\in\cS(\R^d)\setminus\cS_1(\R^d)$, one can define 
 the continuous map $U: \calS_1'(\Rd)\rightarrow \cS'(\Rd)$ by $U(L)(f)= L(\pi(f))$, where for $f \in \calS(\Rd)$, $$\pi(f)=f-c(f)\psi \in \calS_1(\Rd),$$  with $c(f)= \int_{\R^d}f(x)dx  /  \int_{\R^d}\psi(x)dx$. Hence any $\calS_1'(\Rd)$-valued random variable $X$ coincides with the restriction of an $\cS'(\Rd)$-valued random variable $Y$, defined by $Y(f)= X(\pi(f))$, $f \in \calS(\Rd)$. By using the so-defined map $U$ and applying results on $\cS'(\Rd)$, the desired results for $\cS'_1(\Rd)$ follow.
\end{proof}
\subsection{A generalized random ball model}
Now we define the random ball model on $\Rd$. 
Throughout, the operator-scaling is associated to a $d\times d$ real matrix $E$, of which all eigenvalues have strictly positive real parts, denoted by $a_1\geq \cdots\geq a_d>0$.
Let $q = {\rm tr}(E)>0$ be the trace of the matrix $E$.

We  consider the kernel operator defined for $(x,r)\in\R^d\times (0,\infty)$ and $f\in \cS(\R^d)$, by 
\begin{equation}\label{kernel}
T_r^Ef(x):=\int_{\R^d}K_r^E(x,y)f(y)dy \quad\text{ with } K_r^E(x,y):=\mathbf{1}_{B_E(x,r)}(y).
\end{equation}
Here and throughout, $B_E(x,r)$ is the shifted and scaled ``ball'' given by 
\[
B_E(x,r) = x+r^EB,\quad x\in\Rd,\,r>0,
\]
based on a fixed bounded measurable set $B\subset\Rd$ with $0\in B$, $v_B :=\Leb_d(B)\in(0,\infty)$ and $\Leb_d(\partial B) = 0$, where $\Leb_d$ is the Lebesgue measure on $\Rd$. Thus $v_r:= \Leb_d(B_E(x,r)) = r^qv_B$. Note that we keep the name ``random ball'' from the original model but here the set $B$ can be a much more general set than a ball.
We only assume that $B$ is a set of finite perimeter in the sense that
\begin{equation}\label{eq:finitePerimeter}
 {\rm Per}(B):=\sup\left\{\int_B {\rm div}\varphi(x)\, dx \;:\; \varphi\in\calC^1_c(\R^d,\R^d),\, \|\varphi\|_\infty \le 1  \right\} <\infty,
\end{equation}
where $\calC^1_c(\R^d,\R^d)$ is the set of continuously differentiable functions with compact support (e.g.\ $B$ can be any bounded convex set).
According to \cite[Theorem~14]{galerne2011}, \eqref{eq:finitePerimeter} is equivalent to the fact that the covariogram $g_B:\R^d\ni x\mapsto\Leb_d(B\cap (x+B))$ of the set $B$ is Lipschitz, and thus there exists $C>0$ such that
\begin{equation}\label{eq:symdiff}
\Leb_d(B\Delta(x+B)) = 2(g_B(0)-g_B(x))\le C|x|, \;\text{ for all }x\in\R^d.
\end{equation}

\medskip

We first define the model as a collection of random variables indexed by $f\in\calS(\Rd)$, and then prove the existence of regularizations
 afterwards. 
The rescaled random ball field is defined as
\equh\label{eq:XE1}
X^E_\rho(f):=\int_{\Rd\times\R_+\times\R}mT_r^Ef(x){\mathcal N}_\rho( dx, dr,d m), \quad f\in\calS(\Rd),
\eque
where  $\calN_{\rho}$ is a Poisson random measure on $\Rd\times\R_+\times\R$ with intensity  $\lambda(\rho) dx F(dr/\rho)G(d m)$. 
Intuitively, the origins of random balls are distributed as a homogeneous Poisson process with intensity $\lambda(\rho)$, and each random ball is scaled with a random radius with distribution $F_\rho(dr):= F(dr/\rho)$, and is associated with a random weight $m$ with distribution $G$. Positions, scalings and weights are assumed to be independent. There are a few natural assumptions on $F$ and $G$. 
First, 
the expected volume of a random ball is assumed to be finite. That is,
\equh\label{eq:volume}
v_B\int_{\R_+}r^qF( dr)<\infty.
\eque
Moreover, 
we assume that, for some $C_\beta>0$,
\equh\label{eq:f}
F( dr) = p(r) dr \qmwith p(r)\sim C_\beta r^{-1-\beta} \quad\mmas r\to 0^{q-\beta},
\eque
with the convention, $0^\delta = 0$ if $\delta>0$ and $0^\delta = \infty$ if $\delta<0$. This condition is introduced in a compact form for both zoom-in/out scalings to be explained in Section~\ref{sec:3regimes}. It reads as $p(r)$ is regularly varying at $0$ with index $-1-\beta$, only when $\beta<q$; otherwise~\eqref{eq:volume} will be violated. Similarly, 
$p(r)$ is regularly varying at infinity with index $-1-\beta$ when $\beta>q$. 
Next, for the random weights, their distribution $G$ is assumed to be integrable and in the domain of attraction of certain stable distribution $S_\alpha(\sigma,b,0)$ with $\alpha\in(1,2]$, $\sigma>0$, $b\in[-1,1]$. That is, for independent  random variables $M_i$ with common distribution $G$, 
\equh\label{eq:G}
\frac{M_1+\cdots+M_n}{n^{1/\alpha}}\weakto S_\alpha(\sigma,b,0) \quad\text{ with } \alpha\in(1,2].
\eque

A standard reference for stable distributions and processes is \citep{samorodnitsky94stable}.
Under \eqref{eq:volume} and \eqref{eq:G} with $\alpha>1$, the random field \eqref{eq:XE1} is 
well-defined
and integrable. This follows from the fact
\begin{align*}
\mathbb{E}\pp{|X_\rho^E(f)|}
&\le\int_{\Rd\times\R_+\times\R}|m|T_r^E|f|(x)\lambda(\rho) dx F(dr/\rho)G(d m)\\
&\le\lambda(\rho)\rho^q\mathbb{E}(|M|)v_B \|f\|_{_{L^1}}\int_{\R^+}r^qF(dr),
\end{align*}
where $M$ is a real random variable of distribution $G$ and  $\|f\|_{_{L^1}}:=\int_{\R^d}|f(y)|dy$. 
Hence,  a centered rescaled random ball field can be defined by 
 \[
 Y^E_\rho(f):=X^E_\rho(f)-\mathbb{E}\pp{X^E_\rho(f)},\quad f\in\calS(\Rd).
 \]
We come to the generalized random field interpretation of  $X^E_\rho$ and $Y^E_\rho$.

\begin{Prop} Under assumption \eqref{eq:volume}, $X^E_\rho$ and $Y^E_\rho$ are almost surely elements of $\cS'(\R^d)$ and therefore of $\cS_1'(\R^d)$. As a consequence, they admit regularizations in $\cS'(\R^d)$ and therefore in $\cS_1'(\R^d)$.
\end{Prop}

\begin{proof}
Let us quote that $f\mapsto T_r^Ef(x) \in \cS'(\R^d)$, and moreover for all $k\ge 0$,
$$|T_r^Ef(x)|\le \left(\int_{B_E(x,r)}(1+|y|)^{-k} dy\right)\sup_{z\in\R^d}(1+|z|)^k |f(z)|.$$
It follows that,
$$|X^E_\rho(f)|\le \calC_{\rho,k}^E\sup_{z\in\R^d}(1+|z|)^k |f(z)|,$$
with
$$\calC_{\rho,k}^E:=
\int_{\R^d\times\R_+\times\R}|m|\int_{B_E(x,r)}(1+|y|)^{-k} dy{\mathcal N}_\rho(d x,d r,d m).$$
Note that
\begin{align*}
\mathbb{E}\left(\calC_{\rho,k}^E\right)
&=\lambda(\rho)\int_{\R^d\times\R_+\times\R}|m|\int_{B_E(x,r)}(1+|y|)^{-k} dy d xF_\rho(d r)G(d m)\\
&=\lambda(\rho)\rho^q\mathbb{E}(|M|)v_B\int_{\R^+}r^qF(dr)\left(\int_{\R^d}(1+|y|)^{-k} dy\right),
\end{align*}
which is finite under assumption \eqref{eq:volume} as soon as $k>d$.
Hence, $\calC_{\rho,k}^E<\infty$ a.s.~for $k>d$, so that $X^E_\rho\in \cS'(\R^d)$ a.s. Since we also have $f\mapsto \mathbb{E}(X^E_\rho(f))\in \cS'(\R^d)$ by taking expectation in the previous computations,
it follows that the centered field $Y^E_\rho$ is also in $\cS'(\R^d)$ a.s. 
The last part of the proposition is easy since to obtain a regularization in $\cS'(\R^d)$ of a process $X$ which is almost surely element of $\cS'(\R^d)$, it suffices to modify it by setting $X(\omega)\equiv 0$ for the $\omega\in\Omega$ such that $X(\omega)\notin \cS'(\R^d)$, see \cite[p.40]{fernique67processus}.  
\end{proof}
The limit theorems will be based on the characteristic functionals of the centered rescaled random fields
\begin{equation}\label{PsiY}
\cL_{Y_\rho^E}(f)=
\esp\exp\pp{iY_\rho^E (f)} = \exp\pp{\int_{\Rd\times\R_+}\phi_G(T_r^Ef(x))\lambda(\rho)dxF_\rho(dr)},\quad f\in\calS_n(\Rd),
\end{equation}
with
\equh\label{eq:PsiG}
\phi_G(t):= \int (e^{imt} -1 -imt) G(dm)=\cL_M(t)-1-it\esp(M),\quad t\in\R,
\eque
where $M$ is a real random variable of distribution $G$ satisfying \eqref{eq:G}.
\subsection{Zoom-in/out scalings and four regimes}\label{sec:3regimes}

There are two scalings to be considered in the limit theorems. Recall $F_\rho(dr) = F(dr/\rho)$. The case $\rho\to \infty$ corresponds to enlarging the size of each ball, and $\rho\to0$ corresponds to shrinking the size of each ball. We refer to the two scalings as the zoom-in and zoom-out scalings, respectively.

Next, for each type of scaling, there are four qualitatively different regimes. Since the spatial dependence of the random field is essentially determined by overlaps of random balls, heuristically we compute the expected weight of rescaled balls covering a fixed point $y$, denoted by $m(\rho)$, independent from $y$ by stationarity. It is natural to expect $m(\rho)\to c\in[0,\infty]$, and we distinguish $\infty, (0,\infty)$ and $0$ as three different cases. Take the zoom-in scaling case first. Clearly only small balls, say with radius less than $1$ (before the $\rho$-scaling and the constant $1$ is irrelevant) should matter, and we compute
\begin{align*}
m_{\rm in}(\rho) 
&:= \esp\pp{\int_{\Rd\times\R_+\times\R}\indd{y\in B_E(x,r)}\indd{r\leq 1}\calN_\rho(dx,dr,dm)}\\ 
&= \esp(M)\lambda(\rho)v_B\int_0^{1}r^qF_\rho(dr),
\end{align*}
with $$\lambda(\rho)\int_0^{1}r^qF_\rho(dr)\sim \pp{C_\beta \int_0^1r^{q-\beta-1}dr} \lambda(\rho)\rho^\beta \quad\mmas \rho\to\infty.$$
Similarly for the zoom-out case, we compute for number of balls with radius larger than 1,
\begin{align*}
m_{\rm out}(\rho) 
&:= \esp\pp{\int_{\Rd\times\R_+\times\R}\indd{y\in B_E(x,r)}\indd{r> 1}\calN_\rho(dx,dr,dm)} \\
&= \esp(M)\lambda(\rho)v_B\int_1^\infty r^qF_\rho(dr),
\end{align*}
with $$\lambda(\rho)\int_1^{\infty}r^qF_\rho(dr) \sim \pp{C_\beta \int_1^\infty r^{q-\beta-1}dr}\lambda(\rho)\rho^\beta \quad\mmas \rho\to0.$$
The calculations above made use of~\eqref{eq:f}, and also explain why it is a reasonable assumption. Notice that the constant is 
qualitatively
 irrelevant, only the common term $\lambda(\rho)\rho^\beta$ matters, and both cases of scaling can be summarized in the compact form of $\rho\to 0^{\beta-q}$. 

In summary, there are naturally three regimes of interest, characterized by 
\[
\lambda(\rho)\rho^\beta\to\left\{\begin{array}{lc}
\infty & \mbox{(dense regime)},\\
c\in(0,\infty) & \mbox{(intermediate regime)},\\
0 & \mbox{((very-)sparse regime)},\end{array}\right.
\mmas \rho\to 0^{\beta-q},
\]
where within the case $\lambda(\rho)\rho^\beta\to 0$ we shall further identify two sub-regimes, named as sparse and very-sparse regimes in the sequel.
We shall establish limit theorems for different regimes separately, and in each regime our limit theorem and the proof unify both zoom-in and zoom-out scalings (only zoom-out scaling in the very-sparse regime). Furthermore, in each regime we specify two parameters, $\beta$ on the tails of the radius of random balls, and $n$ indicating the zoom-in ($n=1$) and zoom-out ($n=0$) scalings. 


\section{Scaling limits}\label{sec:limit}
We will treat the four regimes separately. 
In each regime, we first introduce the limit field as stochastic integral, then show the existence of its generalized random field version by Minlos--Bochner's theorem and then prove the weak convergence by L\'evy's continuity theorem.  
For easy reading, all the proofs of this section are postponed to 
Section \ref{sec:proof}.
The limit fields appearing here are further investigated in the next sections.

\subsection{Dense regime}
In the dense regime, 
we consider
\[
\lambda(\rho)\rho^\beta\to\infty \mmas \rho\to 0^{\beta-q}, 
\]
and the admissible range of parameters $\beta$ and $n$ is
\equh\label{eq:betan}
\begin{array}{lll}
\beta \in(q,\alpha q) & n = 0 & \mbox{zoom-out scaling},\\
\beta \in(q-a_d,q) & n = 1& \mbox{zoom-in scaling}.
\end{array}
\eque
The following field appears in the limit.
Let $\alpha\in (1,2]$, $\sigma>0$ and $b\in[-1,1]$ be given by \eqref{eq:G} and $C_\beta>0$ be given by \eqref{eq:f}.  Let $M_{\alpha,\beta}$ be an $\alpha$-stable random measure on $\RdRp$ with control measure 
$\sigma^\alpha C_\beta r^{-1-\beta}d rd x$, and constant skewness function $b$.
For $f\in \cS_n(\R^d)$, let us define the stochastic integral
\begin{equation}\label{Zalpha}
Z_{\alpha,\beta}^E(f) := \int_\RdRp T_r^Ef(x)M_{\alpha,\beta}(d r, d x).
\end{equation}
See \citep{samorodnitsky94stable} for more background on stochastic integrals with respect to $\alpha$-stable random measures. 
\begin{Prop}\label{Prop:Zalpha}
Let $\alpha\in (1,2]$. For $\beta,n$ as in~\eqref{eq:betan}, the process $Z_{\alpha,\beta}^E:=\{Z_{\alpha,\beta}^E(f)\}_{f\in\calS_n(\Rd)}$ in~\eqref{Zalpha} is well-defined, has characteristic functional
\begin{align}
\cL_{Z_{\alpha,\beta}^E}(f)& =\exp\ccbb{-C_\beta\sigma^\alpha\int_{\RdRp}| T_r^Ef(x)|^\alpha
\pp{1-ib\epsilon\pp{T_r^Ef(x)}\tan\frac{\alpha\pi}2}r^{-1-\beta}d r d x},\label{eq:charZalpha}
\end{align}
where $\epsilon(s)=\sign(s)$, and admits a version 
with values in $\cS_n'(\R^d)$.
\end{Prop}
Then, we can consider weak convergence in $\cS_n'(\R^d)$ and state the limit theorem in the dense regime.
\begin{Thm}
\label{thm:tightness}
Suppose that the assumptions \eqref{eq:f} and \eqref{eq:G} on $F$ and $G$ hold. Under \eqref{eq:betan},
if $n_1(\rho)\defe\rho^\beta\lambda(\rho)\to\infty$ as $\rho\to 0^{\beta-q}$, then
\[
\frac1{n_1(\rho)^{1/\alpha}}Y_\rho^E \weakto 
 Z_{\alpha,\beta}^E \quad\mmas \rho\to 0^{\beta-q}
\]
in $\cS_n'(\R^d)$. 
\end{Thm}
\begin{Rem}
We let $\{Z_{\alpha,\beta}^E(f)\}_{f\in \cS_n(\Rd)}$ denote the stochastic process indexed by $f$ via \eqref{Zalpha}, and the same notation $Z_{\alpha,\beta}^E$ in Theorem \ref{thm:tightness} for the corresponding version taking values in $\cS_n'(\Rd)$. Similar notations are used for the other regimes.
\end{Rem}
\subsection{Intermediate regime}
In the intermediate regime, we consider
\equh\label{eq:intermediate}
\lambda(\rho)\rho^\beta \to a^{q-\beta}\mmas\rho \to 0^{\beta-q} \qmwith a\in(0,\infty).
\eque
The admissible range of parameters $\beta$ and $n$ is the same \eqref{eq:betan} as in the dense regime.
In this case,  the limit field is represented by a Poisson integral.
For $a\in(0,\infty)$ and $f\in \cS(\R^d)$, we first define
\equh\label{eq:Tra}
 T_{r,a}^E f (x):=\int_{\R^d} \ind_{a^{-E}B_E(x,r)}(y) f(y) dy=T_{r/a}^E f (a^{-E}x)
\eque
and we consider the Poisson integral $J_{a,\alpha,\beta}^E$ defined, for $f\in\S_n(\R^d)$, by
\equh\label{eq:Ja}
 J_{a,\alpha,\beta}^E(f):=\int_{\R^d\times\R_+\times\R_+}m T_{r,a}^E f(x) \tilde{\mathcal{N}}_\beta(dr,dx,dm),
\eque
where $\tilde{\mathcal N}_\beta$ is the compensated Poisson random measure on $\Rd\times\R_+\times\R_+$ with intensity $C_\beta r^{-1-\beta}dxdrG(dm)$, with $C_\beta>0$
given in \eqref{eq:f}.
For more background on Poisson integrals, see for example \citep{kallenberg97foundations}.
\begin{Prop}\label{prop:Ja}
Let $a\in(0,\infty)$. For $\beta,n$ as in~\eqref{eq:betan}, the process
$ J_{a,\alpha,\beta}^E$ in~\eqref{eq:Ja} is well-defined on $\S_n(\R^d)$, has characteristic functional
\begin{equation}\label{characJa}
\cL_{J_{a,\alpha,\beta}^E}(f) = \exp\ccbb{\int_{\Rd\times\R_+}\phi_G(T_{r,a}^E f(x))C_\beta r^{-1-\beta} dr dx},
\end{equation}
where $\phi_G$ is defined by ~\eqref{eq:PsiG} and admits a version 
 with values in $\S_n'(\R^d)$.
\end{Prop}

The limit theorem in the intermediate regime is the following.
\begin{Thm}
\label{thm:tightness2}
Suppose that the assumptions \eqref{eq:f} and \eqref{eq:G} on $F$ and $G$ hold. 
Under \eqref{eq:betan} and \eqref{eq:intermediate},
\[
Y_\rho^E\weakto 
J_{a,\alpha,\beta}^E\quad\mmas \rho\to 0^{\beta-q}
\]
in $\S_n'(\R^d)$.
\end{Thm}

\subsection{Sparse regime}
The sparse regime correspond to
\equh\label{eq:sparse}
\lambda(\rho)\rho^\beta\to 0 \mmas \rho \to 0^{\beta-q} \qmwith \lambda(\rho)\to 0^{q-\beta}.  
\eque
The admissible range of parameters of $\beta$ and $n$ is
\equh\label{eq:betan'}
\begin{array}{lll}
\beta \in(q,\alpha q) & n = 0 & \mbox{zoom-out scaling},\\
\beta \in(q^2/(q+a_d),q) & n = 1& \mbox{zoom-in scaling}.
\end{array}
\eque
Set $\gamma=\beta/q \in (q/(q+a_d),1)\cup(1,\alpha)$.
Let $M\topp1_\gamma$ be a $\gamma$-stable
 random measure having control measure 
 $\sigma_{1,\gamma}\, dx$
  with
\[
\sigma_{1,\gamma} := v_B\pp{ C_\beta q^{-1}\int_{\R_+}(1-\cos(r))r^{-1-\gamma} dr\int_\R|m|^\gamma G(d m)}^{1/\gamma},
\]
and constant skewness function
\[
b_\gamma := -\frac{\int_\R\epsilon(m)|m|^\gamma G(d m)}{\int_\R|m|^\gamma G(d m)}.
\]
We define, for $f\in\S(\R^d)$, 
\[
Z\topp1_\gamma(f) := \int_\Rd f(x)M\topp1_\gamma( dx).
\]
Note that $Z\topp1_\gamma(f)$ is well-defined since $f\in\S(\R^d)\subset L^\gamma(\Rd)$ and its characteristic functional is given by
\begin{equation}\label{Zgammacharac}
 \calL_{Z\topp1_\gamma}(f)=\exp\pp{-\sigma_{1,\gamma}^\gamma
  \int_{\R^d} |\phi(f(x))|^\gamma  \pp{1- ib_\gamma\epsilon(f(x))\tan\frac{\gamma \pi}{2}}dx }.
\end{equation}

\begin{Prop}\label{prop:Zgamma}
For $\alpha \in (1,2]$ and $\gamma\in (q/(q+a_d),1)\cup(1,\alpha)$, the process $Z\topp1_\gamma$ admits a version with values in $\S'_0(\R^d)\subset\S'_1(\R^d)$.
\end{Prop}

\begin{Thm}\label{thm:3}
Suppose that the assumptions \eqref{eq:f} and \eqref{eq:G} on $F$ and $G$ hold.
 Under \eqref{eq:sparse} and \eqref{eq:betan'}
with $n_2(\rho):= (\lambda(\rho)^{1/\beta}\rho)^q$ and $\gamma = \beta/q$, we have
\[
\frac1{n_2(\rho)}{Y_\rho^E}\weakto Z_\gamma\topp1 \quad\mmas \rho\to 0^{\beta-q},
\]
in $\S_n'(\R^d)$.
\end{Thm}

\begin{Rem}
Note that the result in the case $\beta\in(q^2/(q+a_d),q)$ is also new for the isotropic case when $E=I_d$ (the identity matrix). 
\end{Rem}

\subsection{Very-sparse regime}
In this regime, consider 
\equh\label{eq:very_sparse}
\lambda(\rho)\rho^\beta\to 0, \lambda(\rho)\to \infty \mmas \rho\to 0.
\eque
The admissible range of parameters for the very-sparse regime is
\equh\label{eq:betan''}
\begin{array}{lll}
\beta \in (\alpha q, \infty) & n = 0 & \mbox{zoom-out scaling}.
\end{array}
\eque
Let $M\topp2_\alpha$ be a $\alpha$-stable
 random measure having control measure $\sigma_{2,\alpha} dx$ with
\[
 \sigma_{2,\alpha} := \sigma v_B\pp{\int_{\R_+}r^{\alpha q}F( dr)}^{1/\alpha}
\]
and constant skewness function $b$.
For $f\in\S(\R^d)$, we set
\[
Z_\alpha\topp2(f)  := \int_\Rd f(x)M_\alpha\topp2( dx).
\]
\begin{Prop}\label{prop:Zalpha2}
For $\alpha\in(1,2]$, the process $Z_\alpha\topp2$ admits a version with values in $\S'_0(\R^d)$.
\end{Prop}

\begin{Thm}\label{thm:4}
Suppose that the assumptions \eqref{eq:f} and \eqref{eq:G} on $F$ and $G$ hold. 
Under \eqref{eq:very_sparse} and \eqref{eq:betan''}, with $n_3(\rho) := \lambda(\rho)^{1/\alpha}\rho^q$, 
\[
\frac1{n_3(\rho)}{Y_\rho^E}\weakto Z_\alpha\topp 2 \quad\mmas \rho\to0
\]
in $\S'_0(\R^d)$.
\end{Thm}
\subsection{Summary}

For comparison, we summarize in a single statement 
the limit theorems of the different regimes.

\begin{Thm}
Suppose that the assumptions \eqref{eq:f} and \eqref{eq:G} on $F$ and $G$ hold. We have the following weak convergence in $\S_n'(\Rd)$:
\[
\begin{array}{ll}
\displaystyle (\mbox{dense}) \hfill  \frac1{(\rho^\beta\lambda(\rho))^{1/\alpha}}Y_\rho^E \weakto Z_{\alpha,\beta}^E & \mbox{ if }\quad\lambda(\rho)\rho^\beta\to \infty, \beta, n \mbox{ as in \eqref{eq:betan}},\\
\\
(\mbox{intermediate}) \hfill  Y_\rho^E\weakto J_{a,\alpha,\beta}^E & \mbox{ if }\quad \lambda(\rho)\rho^\beta\to a^{q-\beta}\in(0,\infty), \beta, n \mbox{ as in \eqref{eq:betan}},\\
\\
\displaystyle(\mbox{sparse}) \hfill\frac1{(\rho^\beta\lambda(\rho))^{q/\beta}}Y_\rho^E\weakto Z_{\beta/q}\topp1 & \mbox{ if }\quad\lambda(\rho) \rho^\beta\to 0, \lambda(\rho)\to 0^{q-\beta}, \beta, n \mbox{ as in \eqref{eq:betan'}}, \\
\\
\displaystyle(\mbox{very sparse})\ \   \hfill \frac1{\rho^q\lambda(\rho)^{1/\alpha}}Y_\rho^E\weakto Z_\alpha\topp 2 & \mbox{ if }\quad\lambda(\rho)\rho^\beta\to 0, \lambda(\rho)\to \infty, \beta, n \mbox{ as in \eqref{eq:betan''}},
\end{array}
\]
where in all cases the limit is considered as $\rho\to0^{\beta-q}$.
\end{Thm}


\section{Properties of the limit fields}\label{sec:prop}
In this section, we provide some properties of the limit generalized random fields. In the dense and intermediate regimes, the limit generalized random fields explicitly depend on $E$,
and in particular so are their anisotropic properties. For the sparse and very-sparse regimes, all the dependence structures in the discrete models are not observable in the limit, and thus the limit generalized random fields have no specific anisotropic properties. 
Following Dobrushin in \cite{dobrushin79gaussian}, using duality, we can define the following groups of transformations on $\cS_n(\R^d)$:
\begin{itemize}
\item
the group of shift transformations 
${\mathcal T}=\{\tau_h\}_{h\in\R^d}$:
$$\tau_hf(t)=f(t-h),\quad f \in \cS_n(\R^d),~ h\in\R^d,~ t\in\R^d;$$
\item
the group of $E$-operator-scaling  transformations ${\Delta^{E}}=\{\delta^E_c\}_{c\in (0,\infty)}$:
$$\delta^E_c f(t)=c^{-q}f(c^{-E}t),\quad f \in \cS_n(\R^d),~ c\in (0,\infty),~ q=\mbox{tr}(E),~  t\in\R^d.$$
\end{itemize}
Their analogous ${\mathcal T}$, ${\Delta^{E}}$ on $\cS_n'(\R^d)$ are then defined by 
$$\tau_hL(f):=L(\tau_h f), \mbox{ and } \delta^E_c L(f):=L(\delta^E_c f),$$
for $ L\in \cS_n'(\R^d)$. 
Let us note that when the tempered distribution $L$ is given by a function $g$, one recovers
that  $\tau_h L$ is given by the function $g(\cdot+h)$ and $\delta^E_cL$ is given by the function $g(c^E\cdot)$,
thanks to the normalization term. 
\begin{Prop}\label{os:gen}Let $\alpha \in (1,2]$. For  $\beta, n$ as in~\eqref{eq:betan}, the generalized random field 
 $Z_{\alpha,\beta}^E$ in \eqref{Zalpha} is
\begin{itemize}
\item  shift-invariant: $\forall h\in\R^d$, 
$$\tau_h Z_{\alpha,\beta}^E\stackrel{d}{=} Z_{\alpha,\beta}^E,$$
\item
$(E,H)$-operator-scaling for $H=\frac{q-\beta}{\alpha}\in (-q(1-1/\alpha),0)\cup (0,a_d/\alpha)$: $\forall c>0$,
 $${\delta^{E}_c}Z_{\alpha,\beta}^E\stackrel{d}{=} c^{H}Z_{\alpha,\beta}^E.$$ 
\end{itemize}
\end{Prop}
Let us remark that in \cite{dobrushin79gaussian} the first property is called the  stationary $n$-th increments 
while the second one with $E = I_d$ the self-similarity property.
\begin{proof} It suffices to compute the characteristic functional.
Observe that for $f \in \cS_n(\R)$, one has for all $h\in\R^d$,
$$
Z_{\alpha,\beta}^E(\tau_h f)= \int_{\Rd\times\R_+}T_r^Ef(x-h) M_{\alpha,\beta}(dx,dr)  \eqd Z_{\alpha,\beta}^E(f),$$
by 
 a change of variable, while for all $c>0$,
\begin{align*}
Z_{\alpha,\beta}^E(\delta_c^Ef) 
&= \int_{\Rd\times\R_+}T_r^E\delta_c^Ef(x) M_{\alpha,\beta}(dx,d r)\\
&= \int_{\Rd\times\R}T_{r/c}^Ef(c^{-E}x)M_{\alpha,\beta}(dx,d r)\\
& \eqd  c^{(q-\beta)/\alpha}\int_{\Rd\times\R}T_r^Ef(x)M_{\alpha,\beta}(dx,dr) = c^{(q-\beta)/\alpha}Z_{\alpha,\beta}^E(f),
\end{align*}
where the third step also followed from a change of variable argument. 
\end{proof}
For the intermediate case, the limit random field 
$J^E_{a,\alpha,\beta}$ 
in \eqref{eq:Ja} is not $E$-operator-scaling but it has aggregate $E$-operator-scaling property
 as described below, generalizing aggregate similarity property introduced in \cite{bierme10selfsimilar}.
\begin{Prop}
Under the assumption of Theorem~\ref{thm:tightness2}, 
\[
\delta_{k^{1/(q-\beta)}}^E {J}_{a,\alpha,\beta}^E \eqd \summ i1k {J}_{a,\alpha,\beta}^{E,(i)}, \mfa k\in\N,
\]
where $\{{J}_{a,\alpha,\beta}^{E,(i)}\}_{i=1,\dots,k}$ are i.i.d.~copies of ${J}_{a,\alpha,\beta}^{E}$. Furthermore,
\[
\frac1{a^{(q-\beta)/\alpha}}{J}_{a,\alpha,\beta}^E\Rightarrow Z_{\alpha,\beta}^E \quad \mmas a\to 0^{\beta-q}.
\]
\end{Prop}
\begin{proof}
The first part of the proof follows from straightforward calculation of characteristic functionals, with a similar change of variable argument as above. The second part of the proof follows from convergence of characteristic functionals for random variables in the domain of attractions of $S_\alpha(\sigma,b,0)$. The details are omitted.
\end{proof}
At last, remark that in the sparse and very-sparse regimes, the limit random fields have essentially no dependence structure, as the limit random fields are stochastic integrals with respect to stable random measures with {\it constant} control measure on $\Rd$. Thus they inherit no specific anisotropic properties. Nevertheless, for any $E'$ satisfying the same assumption as $E$ with possibly different eigenvalues, writing $q' = {\rm tr}(E')$, it can be shown that
\[
\delta_c^{E'}{Z}_\theta\topp i \eqd c^{\frac{1-\theta}{\theta}q'}{Z}_\theta\topp i
\]
for $i=1,2$ with legitimate parameter $\theta$.

\section{Comments on pointwise representation}\label{sec:represent}
Given a tempered distribution $L\in \cS'(\R^d)$,  it is a natural question to wonder if it may be represented by a Borel measurable function $g$, that is 
$$\forall f \in \cS(\R^d),\quad L(f)=\int_{\R^d}f(t)g(t)dt.$$
We say that a generalized random field $X$ admits a \textit{pointwise reprensentation} if there exists a measurable random field $\{\widehat X(t)\}_{t\in\R^d}$,
meaning as in Definition~9.4.1 of \cite{samorodnitsky94stable} that $\widehat X:\Omega\times \R^d\rightarrow \R$ is a jointly measurable function, such that
\begin{equation*}
X(f)=\int_{\R^d}\widehat X(t) f(t)dt,\quad f\in\cS(\R^d).
\end{equation*}
Conversely, we have the following property.
\begin{Prop}\label{prop:point0}
Let $\{\widehat X(t)\}_{t\in\R^d}$  be a measurable random field.
If there exists $k\in \N$ such that
$$
\int_{\R^d}(1+|t|)^{-k}\mathbb{E}(|\widehat X(t)|)dt<\infty,
$$
then the random field $X$, defined on $\cS_n(\R^d)$ by
$
X(f)=\int_{\R^d}\widehat X(t) f(t)dt
$,
admits a regularization that is a generalized random field.
Moreover, if $\widehat X$ is $(E,H)$-operator-scaling for some $H>0$ in the sense of \eqref{def:op-sc},
then $X$ is $(E,H)$-operator-scaling in the sense of Proposition~\ref{os:gen}.
\end{Prop}
\begin{proof}
Under the assumption, one checks that for all $f\in \cS_n(\R^d)$,
$$\int_{\R^d}|\widehat X(t) f(t)|dt\le {\mathcal C}_k\underset{z\in\R^d}{\sup}(1+|z|)^{k}|f(z)|,$$
where the random constant ${\mathcal C}_k=\int_{\R^d}(1+|t|)^{-k}|\widehat X(t)|dt$ is a.s.\ finite.
This implies that the linear random field $X$ is well-defined and a.s.\ continuous. Hence there exists a regularization of $X$ on $\cS'_n(\R^d)$, see \cite[p.40]{fernique67processus}. The last property of the proposition is straightforward.
\end{proof}

Our centered rescaled random ball field $Y_\rho^E$ defined in Section~\ref{sec:prelim} clearly admits a pointwise representation where $\widehat Y_\rho^E=\widehat X_\rho^E-\esp\widehat X_\rho^E$ and
\[
 \widehat X_\rho^E(t)=\int_{\Rd\times\R_+\times\R}m K_r^E(x,t){\mathcal N}_\rho( dx, dr,d m),\quad t\in\R^d,
\]
with the same Poisson random measure $\mathcal N_\rho$ than in \eqref{eq:XE1}.
Let us consider the limit generalized random field $Z_{\alpha,\beta}^E$ of the dense regime in the case of symmetric weights ($b=0$).
Actually, there are two situations that we treated separately in the following sub-sections.
\subsection{The case $\beta \in (q-a_d,q)$ and $H=\frac{q-\beta}{\alpha}\in (0,a_d/\alpha)$.}
In this case, as proved in Proposition~\ref{prop:point} below, $Z_{\alpha,\beta}^E$ admits a pointwise representation with
\[
\widehat{Z}_{\alpha,\beta}^E(t)=\int_{\R^d\times\R_+}(\mathbf{1}_{B_E(x,r)}(t)-\mathbf{1}_{B_E(x,r)}(0))M_{\alpha,\beta}(dr,dx),\quad t\in\R^d,
\]
satisfying \eqref{def:op-sc} and $M_{\alpha,\beta}$ is the same as in the representation of $Z_{\alpha,\beta}^E$. 
Let us introduce ${\mathcal C}_{_{E}}(t)=\{(x,r); r^{-E}(x-t) \in B\}$ and note that 
\[
\widehat{Z}_{\alpha,\beta}^E(t)=M_{\alpha,\beta}\left({\mathcal C}_{_{E}}(t)\cap {\mathcal C}_{_{E}}(0)^c\right)-M_{\alpha,\beta}\left({\mathcal C}_{_{E}}(t)^c\cap {\mathcal C}_{_{E}}(0)\right),\quad t\in\R^d.
\]
Until here we do not need to assume that $M_{\alpha,\beta}$ has skewness function $b=0$.

With the assumption that  $M_{\alpha,\beta}$ is symmetric, one can check that
\equh\label{eq:symmetric}
\ccbb{\widehat{Z}_{\alpha,\beta}^E(t)}_{t\in\R^d}\overset{f.d.d.}{=}\ccbb{M_{\alpha,\beta}\left(V_t\right)}_{t\in\R^d},
\eque
with $V_t={\mathcal C}_{_{E}}(t)\Delta {\mathcal C}_{_{E}}(0)$. That is, the random field $\what Z_{\alpha,\beta}^E$ has a Chentsov's type representation \citep[Chapter 8]{samorodnitsky94stable}.
In particular, for $H=\frac{q-\beta}{\alpha}\in (0,a_d/\alpha)$ the random field
$\widehat{Z}_{\alpha,\beta}^E$ generalizes isotropic self-similar $(\alpha,H)$-Takenaka random fields (see \citep[p.405]{samorodnitsky94stable}), defined 
by choosing the Euclidean unit ball for $B$ and $E=I_d$, with $a_d=1$. 

The representation \eqref{eq:symmetric} allows us to provide several simulations of our operator-scaling random ball model with symmetric $\alpha$-stable ($S\alpha S$) weights,
following similar ideas as in \cite{BDE13}. See Figures~\ref{Fig1}--\ref{Fig3} in the appendix.

\begin{Prop}\label{prop:point}
For $\beta \in (q-a_d,q)$, there exists a measurable version of $\widehat{Z}_{\alpha,\beta}^E$, also denoted by  $\widehat{Z}_{\alpha,\beta}^E$, such that
${Z}_{\alpha,\beta}^E$ coincides in $\cS'_1(\R^d)$ with the generalized random field
\equh\label{eq:point1}
f\in \cS(\R^d)\mapsto \int_{\R^d}\widehat{Z}_{\alpha,\beta}^E(t)f(t)dt.
\eque
\end{Prop}
\begin{proof}
First note that
$$
\int_{\R^d}|\mathbf{1}_{B_E(x,r)}(t)-\mathbf{1}_{B_E(x,r)}(0)|^\alpha dx=r^qh(r^{-E}t),$$
with $h(z)={\mathcal L}_d(B\Delta (z+B))$. According to \eqref{eq:symdiff}, $h$ satisfies $h(z)\le C(|z|\wedge 1)$ for some constant $C>0$.
It follows that 
\begin{align*}
\int_{\R^d\times\R^+}&|\mathbf{1}_{B_E(x,r)}(t)-\mathbf{1}_{B_E(x,r)}(0)|^\alpha \sigma^\alpha C_\beta r^{-1-\beta}drdx\\
&\le C\sigma^\alpha C_\beta \int_{\R^+} r^q(|r^{-E}t|\wedge 1)r^{-1-\beta}dr \le C\sigma^\alpha C_\beta \int_{\R_+}r^q(\|r^{-E}\|\wedge 1) r^{-1-\beta}dr (1+|t|)\\
& = C_{\alpha,\beta}^E(1+|t|),
\end{align*}
with $C_{\alpha,\beta}^E=C\sigma^\alpha C_\beta \int_{\R^+} (\|r^{-E}\|\wedge 1)r^{q-\beta-1}dr<\infty$ and $\|\cdot\|$ the subordinated norm, since $\beta \in (q-a_d,q)$. 
Hence 
$\what Z_{\alpha,\beta}^E(t)$ is well-defined  and is a  $S\alpha S$ random variable with scale parameter bounded by $\left(C_{\alpha,\beta}^E(1+|t|)\right)^{1/\alpha}$, for every $t\in\Rd$.
According to \citep[Theorem 11.1.1]{samorodnitsky94stable} there exists a measurable version of $\widehat{Z}_{\alpha,\beta}^E$ since
\begin{enumerate}
\item $(t,x,r)\in\R^d\times\R^d\times\R^+\mapsto (\mathbf{1}_{B_E(x,r)}(t)-\mathbf{1}_{B_E(x,r)}(0))\in\R$ is measurable;
\item the control measure $\sigma^\alpha C_\beta r^{-1-\beta}drdx$ is $\sigma$-finite.
\end{enumerate}
Noting that by  \citep[Property 1.2.17]{samorodnitsky94stable}, we have
\equh\label{eq:1st_moment}
\mathbb{E}\pp{|\widehat{Z}_{\alpha,\beta}^E(t)|}\le \mathbb{E}(|S_\alpha|)\left(C_{\alpha,\beta}^E(1+|t|)\right)^{1/\alpha},
\eque
with $S_\alpha$ a $S\alpha S$ random variable of scale parameter $1$, we may define $f\in\cS(\R^d)\mapsto \int_{\R^d} \widehat{Z}_{\alpha,\beta}^E(t)f(t)dt$ that is a.s.\ in $\cS'(\R^d)$, thanks to Proposition \ref{prop:point0}.

Now it remains to show that the right-hand side of \eqref{eq:point1} has the same stable law as $Z_{\alpha,\beta}^E(f) = \int_{\Rd\times\R_+}T_r^Ef(x)M_{\alpha,\beta}(dx,dr)$. For this we recall  that 
\equh\label{eq:change_of_order}
\int\what Z_{\alpha,\beta}(t)f(t)dt \eqd \int_{\Rd\times\R_+}\pp{\int_\Rd\pp{\mathbf{1}_{B_E(x,r)}(t)-\mathbf{1}_{B_E(x,r)}(0)}f(t)dt} M_{\alpha,\beta}(dx,dr),
\eque
provided that 
\[
\int_\Rd|\what Z_{\alpha,\beta}^E(t)|f(t)dt<\infty \mbox{ a.s.},
\]
see \citep[Theorem 11.4.1]{samorodnitsky94stable}.
Since $f$ decays rapidly, the above follows from \eqref{eq:1st_moment} and hence \eqref{eq:change_of_order} holds. To complete the proof, it remains 
to remark that for $f \in \cS_1(\R^d)$, one has
$$\int_{\R^d}(\mathbf{1}_{B_E(x,r)}(t)-\mathbf{1}_{B_E(x,r)}(0))f(t)dt=T_r^Ef(x).$$
\end{proof}

\medskip

\subsection{The case $\beta \in (q,\alpha q)$ and $H=\frac{q-\beta}{\alpha}\in (-q(1-1/\alpha),0)$.} 
In this case, $H<0$ and  we do not have direct pointwise representation, but the limit field $Z_{\alpha,\beta}^E$ can be obtained as the derivative (in the sense of distributions) of a pointwise process.
For all $t\in\R^d$, following the same idea as for the definition of $Z_{\alpha,\beta}^E(f)$ for $f\in\cS(\R^d)$, we can define the random variable
\[
\widecheck Z_{\alpha,\beta}^E(t)=
\epsilon(t_1)\cdots\epsilon(t_d)\int_\RdRp T_r^E\ind_{[0,t]}(x)M_{\alpha,\beta}(d r, d x),
\]
where the random measure $M_{\alpha,\beta}$ is the same as in \eqref{Zalpha} and $[0,t]=\prod_{i=1}^d[0,t_i]$. 
The family $\widecheck Z_{\alpha,\beta}^E=\{\widecheck Z_{\alpha,\beta}^E(t)\}_{t\in\R^d}$ is a measurable random field and,
by successive integrations by parts, we can show that 
$Z_{\alpha,\beta}^E = D^{(1,\ldots,1)} \widecheck Z_{\alpha,\beta}^E$, that is for all $f\in\cS(\R^d)$,
\[
 Z_{\alpha,\beta}^E(f)= (-1)^d\int_{\R^d}\widecheck Z_{\alpha,\beta}^E(t)  D^{(1,\ldots,1)}f(t) dt.
\]
This consideration is analogous to \citep[Theorem~2.6 and Lemma~3.7]{breton11functional} for $E=I_d$ and $\beta>q=d$ in ${\mathcal D}'(\R^d)$ the space of distribution instead of $\cS'(\R^d)$. We thus refer to \cite{breton11functional} for technical details.


\section{Proofs of the main results}\label{sec:proof}

\subsection{Preliminary results} 

The proofs of our limit theorems follow the same scheme as in \cite{bierme10selfsimilar} or \cite{breton09rescaled} to establish the convergence of the characteristic functions.
They use the two following lemmas concerning conditions \eqref{eq:f} and \eqref{eq:G}.
\begin{Lem}[Lemma 2.4 in \cite{bierme10selfsimilar}, Lemma 3.2 in \cite{breton09rescaled}]\label{lem:g}
Under the assumption~\eqref{eq:f}, if $\{g_\rho\}_{\rho>0}$, $g$ are continuous functions on $\R_+$ such that 
\equh
\lim_{\rho\to  0^{\beta-q}}|g(r)-g_\rho(r)| = 0,\label{cond:limg}
\eque
and for some $0<\beta_-<\beta<\beta_+$ there exists a constant $C>0$ such that
\eqnhn
& & |g(r)|\leq C(r^{\beta_-}\wedge r^{\beta_+}),\label{cond:boundg}\\
& & |g_\rho(r)| \leq C(r^{\beta_-}\wedge r^{\beta_+}), \label{cond:boundgrho}
\eqnen
for all $r>0$,
then, 
for $C_\beta$ as in \eqref{eq:f},
\[
\int_{\R_+}g_\rho(r) F_\rho(r) \sim C_\beta\rho^\beta \int_{\R_+} g(r) r^{-1-\beta} dr, \quad\text{as }\rho\to0^{\beta-q}. 
\]
\end{Lem}
\begin{Lem}[Lemma 3.1 in \cite{breton09rescaled}]\label{lem:G} Suppose that $M$ is in the domain of attraction of $S_\alpha(\sigma,b,0)$ for some $\alpha>1$,
$\sigma>0$ and $b\in\R$. Then 
$$\phi_G(t)=\cL_M(t)-1-it\mathbb{E}(M)\sim-|t|^\alpha\phi_{\alpha,b,\sigma}(t), \mbox{ as } t\rightarrow 0,$$
with 
\begin{equation}\label{def:phi}
\phi_{\alpha,b,\sigma}(t)=\sigma^\alpha(1-ib\epsilon(t)\tan(\alpha\pi/2)),
\end{equation}
where $\epsilon(t)=\text{sign}(t)$. 
Furthermore, there exists $C>0$ such that for all $t \in \R$,
\equh\label{eq:phi_G}
|\phi_G(t)|\le C|t|^\alpha.
\eque
\end{Lem}

The key ingredients for our generalized random ball model are
the precise continuity properties of the operators $T_r^E$ stated in
the following proposition. Recall that we write 
$v_r = \Leb_d(B_E(0,r)) = r^qv_B$, $r>0$, and for $\gamma>0$, $\|f\|_{_{L^\gamma}}^\gamma=\int_{\R^d}|f(x)|^\gamma dx$.

\begin{Prop}\label{prop:Tr}
(i) For all $\gamma\in [1,2]$, $r>0$, and $f\in\S(\R^d)$,
\begin{equation}\label{majoS}
\|T_r^E f\|_{_{L^\gamma}}\le v_{r}\|f\|_{_{L^\gamma}},
\end{equation}
and
\begin{equation}\label{majo:u}
 \|T_r^E f\|_{_{L^\gamma}}\le v_{r}^{1/\gamma}\|f\|_{_{L^1}}.
\end{equation}
As a consequence, for $\gamma\in (1,2]$ and $\beta\in(q,\gamma q)$, there exists some constant $C>0$ such that
\equh\label{eq:RT0}
\int_{\R_+}\snn{T_r^Ef}_{_{L^\gamma}}^\gamma r^{-1-\beta}dr \leq C\nn f_{_{L^1\cap L^\gamma}}^\gamma,\quad f\in\calS(\Rd),
\eque
with $\nn f_{_{L^1\cap L^\gamma}}:=\nn f_{_{L^1}} \vee \nn f_{_{L^\gamma}}$.

(ii) For all $\gamma\in [1,2]$, $r>1$, and $f\in\S_1(\R^d)$,
\begin{equation}\label{majoS1}
\|T_r^E f\|_{_{L^\gamma}}^\gamma \le C r^{q-a_d}(|\log r|\vee 1)^{\ell_d-1}\nn f_{_{L^1}}^{\gamma-1}\int_\Rd|y||f(y)|d y,
\end{equation}
where $\ell_d\le d$ is the number of eigenvalues of $E$ having the minimal real part $a_d$ (counted with multiplicities).
As a consequence, for $\beta\in(q-a_d,q)$ there exists a constant $C$ such that 
\equh\label{eq:RT1}
\int_{\R_+}\snn{T_r^Ef}_{_{L^\gamma}}^\gamma r^{-1-\beta}dr \leq 
C \nn f_{_{L^1}}^{\gamma-1}\int_\Rd (1+|y|)|f(y)|dy, \quad f\in\calS_1(\Rd).
\eque
\end{Prop}

\begin{proof}
(i) Note that
$$\|T_r^E f\|_{_{L^1}}:=\int_{\R^d}|T_r^E f(x)|dx\le \int_{\R^d}\int_{\R^d} K_r^E(x,y)|f(y)|dydx,$$
with $K_r^E(x,y)=\mathbf{1}_{B_E(x,r)}(y)$ by \eqref{kernel}.
Hence, by Fubini's theorem,
\begin{equation}\label{L1L1}
\|T_r^E f\|_{_{L^1}}\le v_{r}\|f\|_{_{L^1}}.
\end{equation}
Moreover, 
\[
\|T_r^E f\|_{_{L^2}}^2=\int_{\R^d}|T_r^E f(x)|^2dx\le \int_\Rd v_r\int_\Rd K_r^E(x,y)|f(y)|^2dydx =  v_{r}^2\|f\|_{_{L^2}}^2,
\]
where we first applied the Cauchy--Schwarz inequality, and Fubini's theorem at the end. 
According to the Riesz--Thorin interpolation theorem (see \cite{Bergh1976}),
 combining this with \eqref{L1L1}, we get \eqref{majoS}.
Moreover, since by the Cauchy--Schwarz inequality we also have  
\[
\|T_r^E f\|_{_{L^2}}^2  \leq \int_\Rd\int _\Rd K_r^E(x,y)|f(y)|dy \nn f_{_{L^1}}dx = v_r\|f\|_{_{L^1}}^2,
\] it follows by H\"older's inequality
that, for $p >1$ such that $\gamma=1/p+2(1-1/p)$,
\[
\|T_r^E f\|_{_{L^\gamma}}^\gamma  \le \|T_r^E f\|_{_{L^1}}^{1/p}\|T_r^E f\|_{_{L^2}}^{2(1-1/p)} 
\leq v_{r}^{1/p}\nn f_{_{L^1}}^{1/p}v_{r}^{1-1/p}\nn f_{_{L^1}}^{2(1-1/p)} 
= v_{r}\|f\|_{_{L^1}}^\gamma.
\]
Since $v_{r}=r^q v_B$ with $q=\mbox{tr}(E)$ we can conclude that for
$\beta \in (q,\gamma q)$, by~\eqref{majoS} and~\eqref{majo:u}, 
\[
\int_{\R_+}\|T_r^E f\|_{_{L^\gamma}}^\gamma r^{-\beta-1}dr\leq \pp{(v_B\nn f_{_{L^1}}^\gamma)\vee(v_B^\gamma\nn f_{_{L^\gamma}}^\gamma)}\int_{\R_+} r^{q-\beta-1} \wedge r^{\gamma q-\beta-1}dr.
\]
Therefore we have proved~\eqref{eq:RT0}. 

\medskip

\noindent (ii) The assumption that $f\in \cS_1(\R^d)$ implies that $\int_{\R^d}f(z)dz=0$ so that
$$T_r^E f(x)=\int_{\R^d}\tilde{K}_r^E(x,y)f(y)dy,$$
with $\tilde{K}_r^E(x,y)=\mathbf{1}_{B_E(x,r)}(y)-\mathbf{1}_{B_E(x,r)}(0)$. Then, by H\"older's inequality,
one has
\begin{align*}
\|T_r^E f\|_{_{L^\gamma}}^\gamma 
&=\int_\Rd\abs{\int_\Rd\pp{\ind_{B_E(y,r)}(x) - \ind_{B_E(0,r)}(x)}f(y)d y}^\gamma d x\\
&\leq\|f\|_{_{L^1}}^{\gamma-1}\int_\Rd\pp{\int_\Rd\abs{\ind_{B_E(y,r)}(x) - \ind_{B_E(0,r)}(x)}^\gamma|f(y)|d y} dx.
\end{align*}
Also,
\[
\int_\Rd\abs{\ind_{B_E(y,r)}(x) - \ind_{B_E(0,r)}(x)}^\gamma d x = \Leb_d(B_E(y,r)\triangle B_E(0,r)) 
= r^qh(r^{-E}y) 
\]
with $h(z) = \Leb_d(B_E(0,1)\triangle B_E(z,1))=\Leb_d(B\triangle (z+B))$, that does not depend on $E$. 
By \eqref{eq:symdiff}, $h(y)\leq C|y|$ for all $y\in\Rd$ and
it follows that,
\equh\label{eq:S1}
\|T_r^E f\|_{_{L^\gamma}}^\gamma\leq C\|f\|_{_{L^1}}^{\gamma-1}\int_\Rd r^q|r^{-E}y||f(y)| d y.
\eque
Recall that according to the Jordan decomposition theorem, 
given $E$, there exists an invertible matrix $P$ such that $D = P\inv E P$ has the real canonical form
\[
\pp{
\begin{array}{ccc}
J_1& & 0\\
& \ddots & \\
0 & & J_p
\end{array}},
\]
where $p$ corresponds to the number of distinct real parts of eigenvalues 
and each block matrix $J$ is either
\begin{itemize}
\item [(i)] a Jordan cell matrix of size $\ell$ 
\[
\pp{
\begin{array}{cccc}
a& 0 & & 0\\
1 & a & \ddots & \\
& \ddots & \ddots & 0\\
0 & & 1 & a
\end{array}},
\]
with $a$ a real eigenvalue of $E$, or 
\item [(ii)] a $2\ell\times2\ell$ matrix in form of
\[
\pp{
\begin{array}{cccc}
\Lambda &  & & 0\\
I_2 & \Lambda & & \\
& \ddots & \ddots & \\
0 & & I_2 & \Lambda
\end{array}}
\qmwith 
\Lambda = \pp{\begin{array}{cc}a & b\\ b& a\end{array}} \mand I_2 = \pp{\begin{array}{cc}
1 & 0\\0&1\end{array}},
\]
with $a\pm ib$ ($b\neq 0$) being complex conjugated eigenvalues of $E$. 
\end{itemize}
In either case, for the subordinated norm $\nn\cdot$
 of the Euclidean norm on $\R^d$, for each block $J$ with the corresponding real part of eigenvalue denoted by $a$, it is shown in \citep[Lemma 3.2]{bierme09holder} that 
\[
r^a \leq \nn{r^J}\leq \sqrt{2\ell} e r^a (|\log r|\vee 1)^{\ell-1}, \mfa r>0.
\]
(This is slightly different from \citep[Lemma 3.2]{bierme09holder}, but can be easily established by following the proof carefully.)
Recall that it is assumed that the real parts of eigenvalues of $E$ satisfy $a_1\geq \cdots\geq a_d>0$. Let $\ell_d$ be the size of the Jordan block associated with $a_d$ and note that the other Jordan blocks, if they exist, are associated with a strictly greater real part.
 Then, there exists a constant $C>0$, such that
\[
\nn{r^E} \leq C r^{a_d}(|\log r|\vee 1)^{\ell_d-1}, \mfa r\in(0,1).
\]
Now, it follows from~\eqref{eq:S1} that for $f\in \cS_1(\R^d)$ one has for $r>1$,
\[
\|T_r^E f\|_{_{L^\gamma}}^\gamma \le C r^{q-a_d}(|\log r|\vee 1)^{\ell_d-1}\nn f_{_{L^1}}^{\gamma-1}\int_\Rd |yf(y)|dy.
\]
Hence, for $\beta\in(q-a_d,q)$, $f\in\calS_1(\Rd)$, combining the above inequality for $r>1$ with \eqref{majo:u} for $r\le 1$, we obtain
\begin{align*}
\int_{\R_+}\snn{T_r^Ef}_{_{L^\gamma}}^\gamma r^{-1-\beta}dr \leq 
& \,C\pp{\nn f_{_{L^1}}^{\gamma-1}\int_\Rd (1+|y|)|f(y)|dy} \\
& \times\int_{\R_+}r^{-1-\beta+q}\wedge \pp{r^{-1-\beta+q-a_d}(|\log r|\vee 1)^{\ell_d-1}}dr,
\end{align*}
which proves \eqref{eq:RT1}. 
\end{proof}

\subsection{Dense regime}

\begin{proof}[Proof of Proposition~\ref{Prop:Zalpha}]
First, the stochastic integral $Z_{\alpha,\beta}^E(f)$ in \eqref{Zalpha} is well-defined as soon as 
\begin{equation*}
\int_\RdRp |T_r^Ef(x)|^\alpha r^{-1-\beta}d rd x= \int_{\R_+}\snn {T_r^Ef}_{_{L^\alpha}}^\alpha r^{-1-\beta} dr<\infty
\end{equation*}
and this condition follows from Proposition~\ref{prop:Tr}, with $\gamma=\alpha$, $\beta$, $n$ as in ~\eqref{eq:betan}.
It is well known (see \cite[Chap.~3]{samorodnitsky94stable}) that the characteristic functional $\cL_{Z_{\alpha,\beta}^E}$ of $Z_{\alpha,\beta}^E$ on $\cS_n(\R^d)$ is given by \eqref{eq:charZalpha}.
Now, according to Theorem~\ref{thm:version}, to prove the existence of a generalized-random-field version of $Z_{\alpha,\beta}^E$, it suffices to prove that $\cL_{Z_{\alpha,\beta}^E}$ is continuous on $\cS_n(\R^d)$, that is, for all $\{f_k\}_{k\in\N}$ and $f$ in $\calS_n(\Rd)$ such that $f_k\rightarrow f$ in $\cS_n(\R^d)$, $\lim_{k\to\infty}\cL_{Z_{\alpha,\beta}^E}(f_k)=\cL_{Z_{\alpha,\beta}^E}(f)$. 
This shall follow from the convergence in distribution of the random variables $Z_{\alpha,\beta}^E(f_k-f)$ to $0$ as $k\to\infty$, or equivalently from
\[
\lim_{k\to\infty}\int_{\Rd\times\R_+}\snn {T_r^E(f_k-f)}_{_{L^\alpha}}^\alpha r^{-1-\beta}dr = 0.
\] 
By \eqref{eq:RT0} and \eqref{eq:RT1} of Proposition~\ref{prop:Tr} with $\gamma=\alpha$, this is straightforward, since $f_k-f\rightarrow 0$ in $\calS_n(\Rd)$
clearly implies that the upper bounds also tend to $0$.
\end{proof}

\begin{proof}[Proof of Theorem~\ref{thm:tightness}]
Note that, by Theorem~\ref{thm:levy},
the result follows from the pointwise convergence of the characteristic functional. Further, by \eqref{PsiY}, we clearly have for $f\in\calS_n(\Rd)$,
$$
\cL_{n_1(\rho)^{-1/\alpha}Y_\rho^E}(f)=
 \exp\pp{\int_{\Rd\times\R_+}\phi_G\pp{\frac{T_r^Ef(x)}{n_1(\rho)^{1/\alpha}}}\lambda(\rho)dxF_\rho(dr)}.
$$
Since $n_1(\rho)\to\infty$, by Lemma \ref{lem:G},
$$
\phi_G\pp{\frac{T_r^Ef(x)}{n_1(\rho)^{1/\alpha}}}\sim \frac1{n_1(\rho)}|T_r^Ef(x)|^\alpha\phi_{\alpha,b,\sigma}(T_r^Ef(x)),\quad\text{as }\rho\to0^{\beta-q},
$$
for $\phi_{\alpha,b,\sigma}$ defined in \eqref{def:phi}.
Hence, under \eqref{eq:f}, one can apply Lemma \ref{lem:g} to prove that
$$\cL_{n_1(\rho)^{-1/\alpha}Y_\rho^E}(f)\rightarrow \cL_{Z_{\alpha,\beta}^E}(f).$$
Indeed, recall the uniform bound \eqref{eq:phi_G} on $\phi_G$ and, thanks to Proposition~\ref{prop:Tr}, the fact that
for $n=0$,
\[
\|T_r^Ef\|_{_{L^\alpha}}^\alpha \le C_E\|f\|_{_{L^1\cap L^\alpha}}^\alpha (r^q\wedge r^{\alpha q}), 
\]
and for $n=1$,
\[
\|T_r^Ef\|_{_{L^\alpha}}^\alpha \le  C_E\|f\|_{_{L^1}}^{\alpha-1}\left(\int_{\R^d}(1+|y|)|f(y)|dy\right) (r^q\wedge r^{q-a_p}|\log(r)|^{d-1}).
\]
We can then apply Lemma~\ref{lem:g} with $g_\rho(r)=n_1(\rho)\int_{\R^d}\phi_G(n_1(\rho)^{-1/\alpha}T_r^Ef(x))dx$ to both cases $\beta\in(q,\alpha q)$ and $\beta\in(q-a_d,q)$.
\end{proof}

\subsection{Intermediate regime}
\begin{proof}[Proof of Proposition~\ref{prop:Ja}]
Recall that the Poisson integral $J_{a,\alpha,\beta}^E(f)$ in \eqref{eq:Ja} is well-defined as soon as 
\[
\int_{\R^d\times\R_+\times\R_+} \pp{|m T_{r,a}^E f(x)|\wedge|m T_{r,a}^E f(x)|^2} r^{-1-\beta}dxdrG(dm)<\infty.
\]
Let us remark that
$$|m T_{r,a}^E f(x)|\wedge|m T_{r,a}^E f(x)|^2\le |m T_{r,a}^E f(x)|^\gamma,$$
for any $\gamma\in [1,2]$. Hence, for 
$\beta\in(q-a_d,q)\cup (q,\alpha q)$, 
choosing $\gamma\in [1,\alpha)$ such that
$\beta\in (q-a_d,\gamma q)$, one has
$$\int_{\R^d\times\R_+\times\R_+} |m T_{r,a}^E f(x)|^\gamma\, r^{-1-\beta}dxdrG(dm)\le \mathbb{E}(|M|^\gamma)\int_{\R^+}\|T_{r,a}^E f\|_{_{L^\gamma}}^\gamma r^{-1-\beta}dr<\infty,$$
in view of Proposition \ref{prop:Tr}, since $\|T_{r,a}^E f\|_{_{L^\gamma}}^\gamma=a^q\|T_{r/a}^E f\|_{_{L^\gamma}}^\gamma$ (see \eqref{eq:Tra}).
It follows that the Poisson integral $J_{a,\alpha,\beta}^E(f)$ is well-defined for all $f \in\cS_n(\R^d)$ and the characteristic functional $\cL_{J_{a,\alpha,\beta}^E}$ of $J_{a,\alpha,\beta}^E$ is given by \eqref{characJa}.

Again, to show  the existence of a version of $J_{a,\alpha,\beta}^E$ with values in $\S_n'(\R^d)$, using Theorem~\ref{thm:version}, it is sufficient to prove that the characteristic functional $\cL_{J_{a,\alpha,\beta}^E}$ is continuous on $\S_n(\R^d)$.
Let $\beta\in(q-a_d,q)\cup (q,\alpha q)$ and assume that $f_k\to 0$ in $\S_n(\R^d)$. We will show that $J_{a,\alpha,\beta}^E(f_k)$ converges in $L^\gamma$ to $0$, which is sufficient to prove the continuity of $\cL_{J_{a,\alpha,\beta}^E}$. Actually, following the proof of Proposition 3.1 in  \cite{breton11functional}, we can bound $\gamma$-moments of the real 
random variable $J_{a,\alpha,\beta}^E(f)$ for $f \in \S_n(\R^d)$. Since  $J_{a,\alpha,\beta}^E(f)$ is centered, for $\gamma\in [1,\alpha)$, following \cite[p.461]{Gaigalas06} and using Lemma~2 and Lemma~4 of 
\cite{Esseen65}, 
$$\esp\pp{|J_{a,\alpha,\beta}^E(f)|^\gamma}
  \le  A(\gamma)\int_0^\infty\pp{1-\abs{\cL_{J_{a,\alpha,\beta}^E}(\theta f)}^2}\theta^{-1-\gamma}d \theta,$$
with $A(\gamma) := (\int_0^\infty(1-\cos x)x^{-1-\gamma} d x)\inv<\infty$. But
\[
\abs{\cL_{J_{a,\alpha,\beta}^E}(\theta f)}\ge \exp\left(-C|\theta|^\alpha\int_{\Rd\times\R^+}|T_{r,a}^E f(x)|^\alpha C_\beta r^{-1-\beta} dr dx\right),
\]
using the upper bound on $|\phi_G|$ given \eqref{eq:phi_G}.
It follows that for $\gamma\in[1,\alpha)$ one has
\begin{align*}
\esp\pp{|J_{a,\alpha,\beta}^E(f)|^\gamma}  &\le  A(\gamma)\int_0^\infty(1-\exp\left(-2C|\theta|^\alpha\int_{ \R^+}\|T_{r,a}^Ef\|_{_{L^\alpha}}^\alpha C_\beta r^{-1-\beta} dr \right)\theta^{-1-\gamma}d \theta\\
&\le   A(\gamma) A(\alpha,\gamma)\left(C \int_{ \R^+}\|T_{r,a}^Ef\|_{_{L^\alpha}}^\alpha C_\beta r^{-1-\beta} dr\right)^{\gamma/\alpha},
\end{align*}
with $ A(\alpha,\gamma):=\int_0^\infty(1-\exp(-s^\alpha))s^{-1-\gamma}ds<\infty$. Hence the result follows from Proposition~\ref{prop:Tr}
since $\|T_{r,a}^Ef\|_{_{L^\alpha}}^\alpha=a^q\|T_{r/a}^Ef\|_{_{L^\alpha}}^\alpha$.
\end{proof}

\begin{proof}[Proof of Theorem~\ref{thm:tightness2}]
Again, by Theorem \ref{thm:levy}, the result follows from the convergence of the characteristic functionals. Observe that, 
\begin{align*}
\cL_{J_{a,\alpha,\beta}^E}(f)
& = \exp\ccbb{\int_{\Rd\times\R_+}\phi_G(T_{r,a}^E f(x))C_\beta r^{-1-\beta} dr dx}\\
& = \exp\ccbb{C_\beta\int_{\Rd\times\R_+}\phi_G(T_{s}^E f(y))a^{q-\beta}s^{-1-\beta}\d s dy}
\end{align*}
by the changes of variables $y=a^{-E}x$ and  $s=r/a$. 
The rest of the proof can be done similarly as for Theorem~\ref{thm:tightness}, starting from \eqref{PsiY} and applying Lemma~\ref{lem:g} with $g(r)=g_\rho(r)=\int_{\R^d}\phi_G(T_r^Ef(x))dx$ and the help of Proposition~\ref{prop:Tr}.
\end{proof}

\subsection{Sparse regime}

\begin{proof}[Proof of Proposition~\ref{prop:Zgamma}]
Using Theorem~\ref{thm:version}, it is sufficient to prove that $Z\topp1_\gamma(f_k)$ converges in distribution to $0$ when $f_k\to 0$ in $\S(\R^d)$. This last assertion is obvious since convergence in $\S(\R^d)$ implies convergence in $L^\gamma(\R^d)$.
\end{proof}

To prove Theorem~\ref{thm:3}, we consider the maximal function $f^*$ associated to a function $f$ of $\S(\R^d)$,
\[
 f^*(x):=\sup_{r>0}\frac{1}{r^q v_B}\int \ind_{B_E(x,r)}(y)|f(y)| dy, \quad x\in\R^d,
\]
and we shall need the following lemma.
\begin{Lem}\label{lem:maxfunct}
For all $f\in\S(\R^d)$ and all $\alpha>1$, $f^*\in L^\alpha(\R^d)$.
\end{Lem}
\begin{proof}
By Lemma~6.1.5 in \citet{meerschaert04limit}, there exists a norm $\|\cdot\|_0$ on $\R^d$ such that the mapping $(0,\infty)\times\{x\in\R^d\mid\|x\|_0=1\}\to \R^d\setminus\{0\}$, $(t,\theta)\mapsto t^E\theta$, is a homeomorphism. Further, the function $t\mapsto \|t^E x\|_0$ is increasing for all $x\in\R^d$. Thus, any $x\in\R^d\setminus\{0\}$ can be uniquely written as $x=\tau(x)^E\theta(x)$ with $\tau(x)>0$ and $\|\theta(x)\|_0=1$. The function $\tau$ is a continuous function that can be extended to $\R^d$ by setting $\tau(0)=0$.
By Lemma~2.2 in \citet{bierme07operator}, one can find $\kappa\ge 1$ such that
\begin{equation}\label{quasi:dist}
\tau(x+y)\le \kappa\left(\tau(x)+\tau(y)\right).
\end{equation}
Therefore we can introduce the function $\delta(x,y)=\tau(y-x)$, $x,y\in\R^d$, which is a quasi-distance on $\R^d$.
We also introduce the sets
\equh\label{eq:CE}
 C_E(x,r)=\{y\in\R^d\mid\delta(x,y)<r\}, \quad r>0.
\eque
Since $B$ is a bounded subset of $\R^d$, we can find a real $r_0>0$ such that $B\subset C_E(0,r_0)$.
With no loss of generality we assume that $r_0=1$ and we denote $C:=C_E(0,1)$. Thus $C_E(x,r)=x+r^EC$ for all $x\in\R^d$ and $r>0$, and $B_E(x,r)\subset C_E(x,r)$. We infer that for all $x\in\R^d$, 
\[
 f^*(x)\le \frac{v_C}{v_B} \sup_{r>0}\frac{1}{r^q v_C}\int \ind_{C_E(x,r)}(y)|f(y)| dy.
\]
The desired result is now a consequence of Theorem~1 and Example~2.4 in \citet{stein93harmonic}. 
\end{proof}

\begin{proof}[Proof of Theorem~\ref{thm:3}]
By Theorem~\ref{thm:levy}, it is sufficient to prove the convergence of the characteristic functionals.
The characteristic functional of $n_2(\rho)\inv Y_\rho^E$ is given by, recalling that $F_\rho(dr)= F(dr/\rho)$,
\begin{align*}
\cL_{n_2(\rho)\inv Y_\rho^E}(f)
&=\exp\ccbb{\int_\RdRp\lambda(\rho)\phi_G\pp{\frac{T_r^Ef(x)}{n_2(\rho)}} dxF_\rho( dr)}\\
&= \exp\ccbb{\int_{\Rd\times\R_+}\lambda(\rho)\phi_G\pp{\frac{T_{n_2(\rho)^{1/q}r}^Ef(x)}{n_2(\rho)}} dxF_{\rho n_2(\rho)^{-1/q}}( dr)}.
\end{align*}
We shall show that
\begin{multline}\label{eq:2.16}
\int_{\Rd\times\R_+}\lambda(\rho)\phi_G\pp{\frac{T_{n_2(\rho)^{1/q}r}^Ef(x)}{n_2(\rho)}} dxF_{\rho n_2(\rho)^{-1/q}}( dr)\\
\rightarrow C_\beta \int_{\Rd\times\R_+}\phi_G(f(x)v_B r^q)r^{-1-\beta} dr dx\quad\mmas \rho\to0^{\beta -q} \mbox{ and } \lambda(\rho)\to0^{q-\beta}.
\end{multline}
From this, we infer that
\[
 \cL_{n_2(\rho)\inv Y_\rho^E}(f)\to  \exp\ccbb{C_\beta\int_{\Rd\times\R_+}\phi_G(f(x)v_B r^q)r^{-1-\beta} dr dx}=\calL_{Z_\gamma^{(1)}}(f),
\]
for $\calL_{Z_\gamma^{(1)}}(f)$ given in \eqref{Zgammacharac}, which completes the proof. The last equality above is obtained by following the same lines as in \citep[pages 3650--3651]{breton09rescaled}.

To prove \eqref{eq:2.16}, recalling that $\lambda(\rho)\rho^\beta n_2(\rho)^{-\beta/q}=1$, it suffices to check the conditions of Lemma~\ref{lem:g} for
\[
g_\rho(r) := \int_\Rd\phi_G\pp{\frac{ T_{n_2(\rho)^{1/q}r}^Ef(x)}{n_2(\rho)}} dx \quad\text{ and }\quad
g(r) := \int_\Rd\phi_G(f(x)v_Br^q) dx.
\]
First, remark that for $f\in\S_n(\R^d)$, 
\[
\frac{ T_{n_2(\rho)^{1/q}r}^Ef(x)}{n_2(\rho)} \underset{\rho\rightarrow 0^{\beta-q}}{\longrightarrow }v_Br^qf(x)
\]
for $ dx$-almost all $x$, so that
\[
\phi_G\left(\frac{ T_{n_2(\rho)^{1/q}r}^Ef(x)}{n_2(\rho)}\right)\underset{\rho\rightarrow 0^{\beta-q}}{\longrightarrow }\phi_G(v_Br^qf(x))
\]
for $ dx$-almost all $x$ by continuity of $\phi_G$. But, by Lemma~\ref{lem:G},
\[
\left|\phi_G\left(\frac{ T_{n_2(\rho)^{1/q}r}^Ef(x)}{n_2(\rho)}\right)\right|\le 
C\left|\frac{ T_{n_2(\rho)^{1/q}r}^Ef(x)}{n_2(\rho)}\right|^\alpha
\le C(v_Br^q)^\alpha f^*(x)^\alpha.
\]
Since $f^*$ belongs to $L^\alpha(\Rd)$ by Lemma~\ref{lem:maxfunct}, Condition~\eqref{cond:limg} follows by Lebesgue's theorem.

Next, for Condition \eqref{cond:boundg}, we deal with the cases $n=0$ and $n=1$ separately.
Now, since $|\phi_G(u)|\leq C(|u|\wedge|u|^\alpha)$ 
and $f\in L^1(\Rd)\cap L^\alpha(\Rd)$,
\[
|g(r)|\leq C\int_{\Rd}|f(x)v_Br^q|\wedge |f(x)v_Br^q|^\alpha dx \leq C(\|f\|_{_{L^1}}v_B \vee \|f\|_{_{L^\alpha}}^\alpha v_B^\alpha) (r^q\wedge r^{\alpha q}).
\]
This establishes Condition~\eqref{cond:boundg} for $\beta \in (q,\alpha q)$ and $n=0$ with $\beta_- = q$ and $\beta_+ = \alpha q$.
Next, when $f\in \calS_1(\R)$, remark that
$$g(r)=\int_\Rd\phi_G(f(x)v_Br^q) dx=\int_\Rd\tilde{\phi}_G(f(x)v_Br^q) dx,$$
with $\tilde{\phi}_G(u)=\int (e^{imu}-1)G(dm)$ so that now $|\tilde{\phi}_G(u)|\le C(1 \wedge |u|^\delta)$ for any $\delta\in (0,1]$.
Hence 
\[
|g(r)|\leq Cv_B^\delta \|f\|^\delta_{_{L^\delta}}r^{q\delta}.
\]
Choosing $\delta=q/(q+a_d)\in (0,1)$ and $\delta=1$ respectively, we infer that for $n=1$, Condition \eqref{cond:boundg} holds for $\beta\in (q^2/(q+a_d),q)$ with $\beta_- = q^2/(q+a_d)$ and $\beta_+ = q$, respectively.

It remains to prove that \eqref{cond:boundgrho} holds. We first consider $\beta\in(q,\alpha q)$. 
Using $|\phi_G(u)|\leq C|u|$ and \eqref{majo:u},
\equh\label{eq:beta=q}
\left|g_\rho(r)\right|  \le  C\frac{1}{n_2(\rho)} \|T_{n_2(\rho)^{1/q}r}^Ef\|_{_{L^1}}
\le  C\|f\|_{_{L^1}}r^q.
\eque
Then, using $|\phi_G(u)|\leq C|u|^\alpha$, we can write 
\[
\left|g_\rho(r)\right|
\le Cr^{\alpha q}\int_{\R^d} \left|\frac{T_{n_2(\rho)^{1/q}r}^Ef(x)}{n_2(\rho) r^q}\right|^\alpha dx
\le C\|f^*\|_{_{L^\alpha}}^\alpha r^{\alpha q},
\]
that finishes to prove \eqref{cond:boundgrho} when $\beta\in (q,\alpha q)$.
Finally, when $\beta\in (q^2/(q+a_d),q)$ and  $f \in\S_1(\R^d)$, we write
$g_\rho=g_\rho^{(1)}+g_\rho^{(2)}$, with
\[
g_\rho^{(1)}:=\int_{\tau(x)\le 2\kappa n_2(\rho)^{1/q}r}\tilde{\phi}_G\left(\frac{T_{n_2(\rho)^{1/q}r}^Ef(x)}{n_2(\rho)}\right) dx
\]
and
\[
g_\rho^{(2)}:=\int_{\tau(x)> 2\kappa n_2(\rho)^{1/q}r}\tilde{\phi}_G\left(\frac{T_{n_2(\rho)^{1/q}r}^Ef(x)}{n_2(\rho)}\right) dx,
\]
where $\kappa\ge 1$ comes from the quasi-triangular inequality given in \eqref{quasi:dist}. With this choice
we may write for any $z \in (n_2(\rho)^{1/q}r)^EB$, 
$$\tau(x)\le \kappa\left(\tau(x+z)+\tau(z)\right)\le \kappa\left(\tau(x+z)+n_2(\rho)^{1/q}r\right),$$
where, with no loss
of generality, we have again assumed that $B\subset C_E(0,1)$ (recall \eqref{eq:CE}).
It follows that 
 $\tau(x+z)>\frac{1}{2\kappa}\tau(x)$ for any $z \in (n_2(\rho)^{1/q}r)^EB$ and $x$ such that $\tau(x)> 2\kappa n_2(\rho)^{1/q}r$.  Since $f$ is rapidly decreasing,
we get for  $N\ge 1$,
\begin{eqnarray*}
\left|\frac{1}{n_2(\rho)} T_{n_2(\rho)^{1/q}r}^Ef(x)\right|&\le & \int_{\R^d}\mathbf{1}_{r^EB}(z)\left|f(x+n_2(\rho)^{E/q}z)\right|dz\\
&\le & C\int_{\R^d}\mathbf{1}_{r^EB}(z)\left(1 +\tau(x+n_2(\rho)^{E/q}z)\right)^{-N}dz\\
&\le & Cv_Br^q\left(1 +\tau(x)\right)^{-N},
\end{eqnarray*}
where here and below, the constant $C=C(f)$ does not depend on $r$ and $\rho$.
Using that $|\tilde{\phi}_G(u)|\leq C|u|^\delta$ for $\delta \in (0,1]$, choosing $N=N(\delta,q)$ such that $N\delta>q+1$, it follows that
\begin{equation}\label{eq:g2}
\left|g_\rho^{(2)}(r)\right| 
\le Cr^{q\delta}\int_{\R^d} \left(1 +\tau(x)\right)^{-N\delta}dx
\le Cr^{q\delta}.
\end{equation}
Moreover, 
\begin{eqnarray*}
\left|g_\rho^{(1)}(r)\right| &\le & Cn_2(\rho)^{-\delta}\int_{\tau(x)\le C n_2(\rho)^{1/q}r}\left|T_{n_2(\rho)^{1/q}r}^Ef(x)\right|^\delta dx\\
&\le & Cn_2(\rho)^{-\delta}\|T_{n_2(\rho)^{1/q}r}^Ef\|_{_{L^{p\delta}}}^{\delta} \left(n_2(\rho)r^q\right)^{1-1/p},
\end{eqnarray*}
by H\"older's inequality for $p>1$. 
When $n_2(\rho)^{1/q}r\le 1$, we use \eqref{majo:u} with $p\delta \in [1,2]$. It follows that
\begin{equation}\label{eq:g1}
\left|g_\rho^{(1)}(r)\right|\le C n_2(\rho)^{-\delta}\left(n_2(\rho)r^q\right)^{1/p} \times \left(n_2(\rho)r^q\right)^{1-1/p}\le C n_2(\rho)^{1-\delta}r^q \le Cr^{q\delta},
\end{equation}
since $n_2(\rho)\le r^{-q}$.
When $n_2(\rho)^{1/q}r>1$, we use \eqref{majoS1} for $p\delta \in [1,2]$. 
By the assumption that $\beta>q^2/(q+a_d)$, we can choose $b\in (0,a_d)$ such that $\beta>q^2/(q+b)$ and
\[
\|T_{n_2(\rho)^{1/q}r}^Ef\|_{_{L^{p\delta}}}^{\delta}\le C\left(n_2(\rho)^{1/q}r\right)^{(q-b)/p},
\]
by \eqref{majoS1} since $b<a_d$.
Hence,
\[
\left|g_\rho^{(1)}(r)\right|\le C  n_2(\rho)^{-\delta+1-b/qp}r^{q-b/p}.
\]
Now we can choose $\delta=q/(q+b)\in (0, 1)$ and $p=(1+b/q)>1$ such that $\delta p=1$ and
\[
\left|g_\rho^{(1)}(r)\right|\le C r^{q-b/(1+b/q)}=Cr^{q^2/(q+b)}.
\]
Combining with the previous bounds \eqref{eq:g2} and \eqref{eq:g1} for the same $\delta=q/(q+b)$, we get
\[
\left|g_\rho(r)\right|\le C r^{q^2/(q+b)},
\]
and we have that \eqref{cond:boundgrho} holds with $\beta_- = q^2/(q+b)$ and $\beta_+ = q$ (which we have shown in \eqref{eq:beta=q} when considering the case $\beta\in(q,\alpha q)$).
We have thus proved~\eqref{eq:2.16} and the theorem.
\end{proof}

\subsection{Very-sparse regime}
Proposition~\ref{prop:Zalpha2} can be obtained as before using Theorem~\ref{thm:version}. The proof
of Theorem~\ref{thm:4} is similar to the one of Theorem~\ref{thm:3} (see also the proof of \citep[Theorem 2.19]{breton09rescaled}). The details of this part are thus omitted.


\appendix
\section{Illustrations}

We provide several simulations of our operator-scaling random ball model, obtained by following similar ideas as in \cite{BDE13}. For the sake of simplicity we choose
$E=\mbox{diag}(a_1,a_2)$ with $a_1\ge a_2:=1$ and $\beta \in (q-a_d,q)=(a_1,a_1+1)$. 
\begin{figure}[ht!]\centering
\begin{tabular}{ccc}
\includegraphics[height=4.15cm, width=5cm]{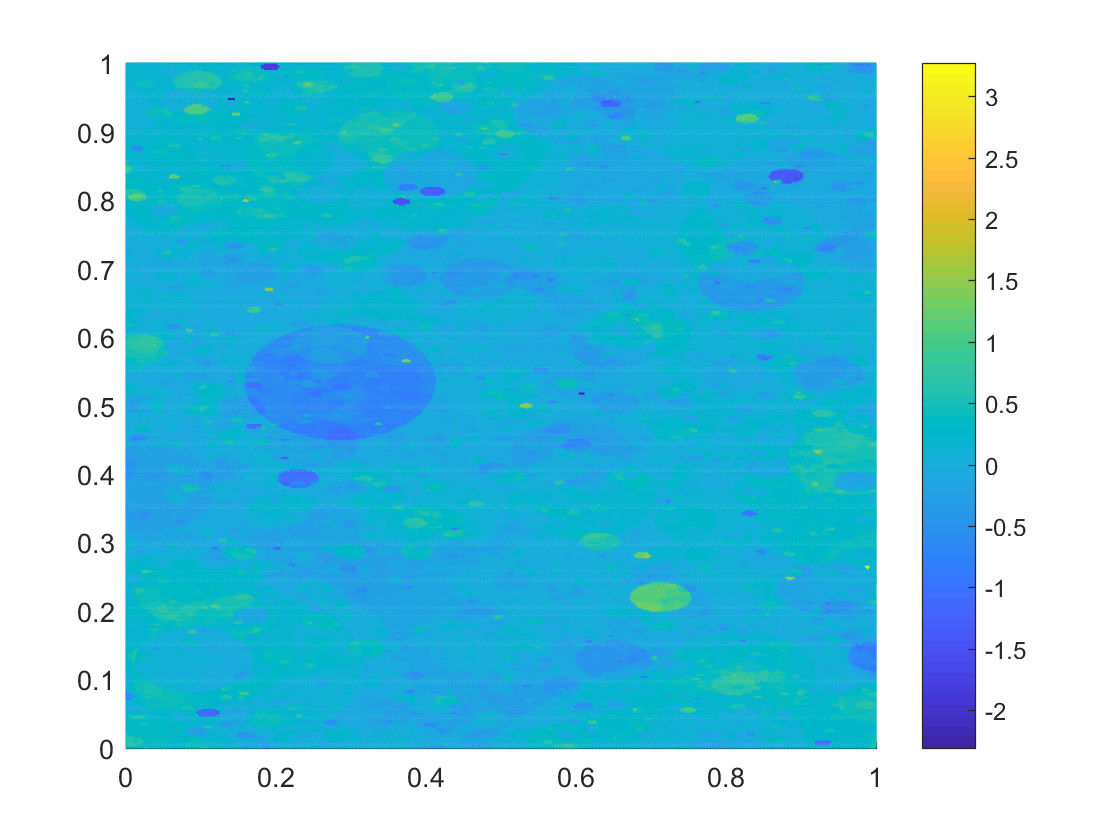}
&
\includegraphics[height=4.15cm, width=5cm]{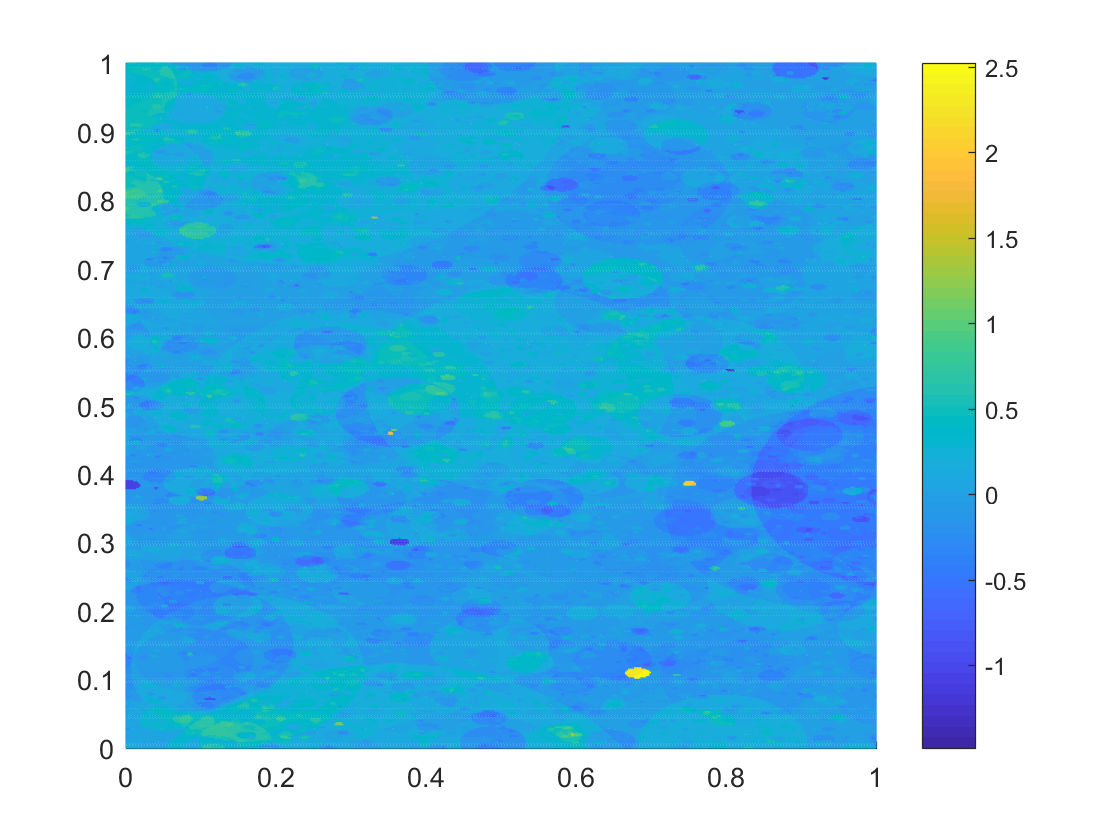}
&
\includegraphics[height=4.15cm, width=5cm]{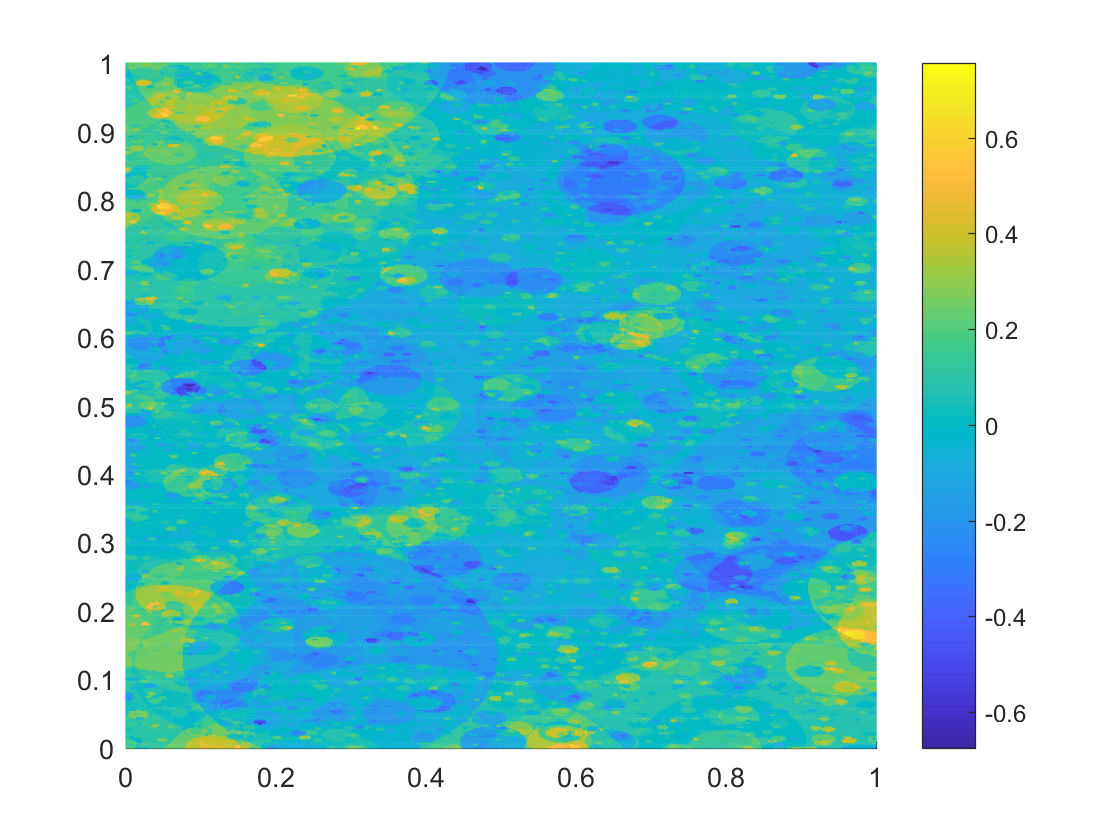}\\
$\alpha=1.7$&$\alpha=1.9$&$\alpha=2$
\end{tabular}
\vspace{-0.25cm}
\caption{\footnotesize{\textbf{Operator-scaling random ball}
 with $a_1=1.2$ and $\beta=1.6$: 
the set $B$ is an Euclidean ball,  the weights vary according to a 
$S\alpha S(\sigma)$ distribution with $\sigma=0.1$.} }\label{Fig1}
\end{figure}

\begin{figure}[ht!]\centering
\begin{tabular}{ccc}
\includegraphics[height=4.15cm, width=5cm]{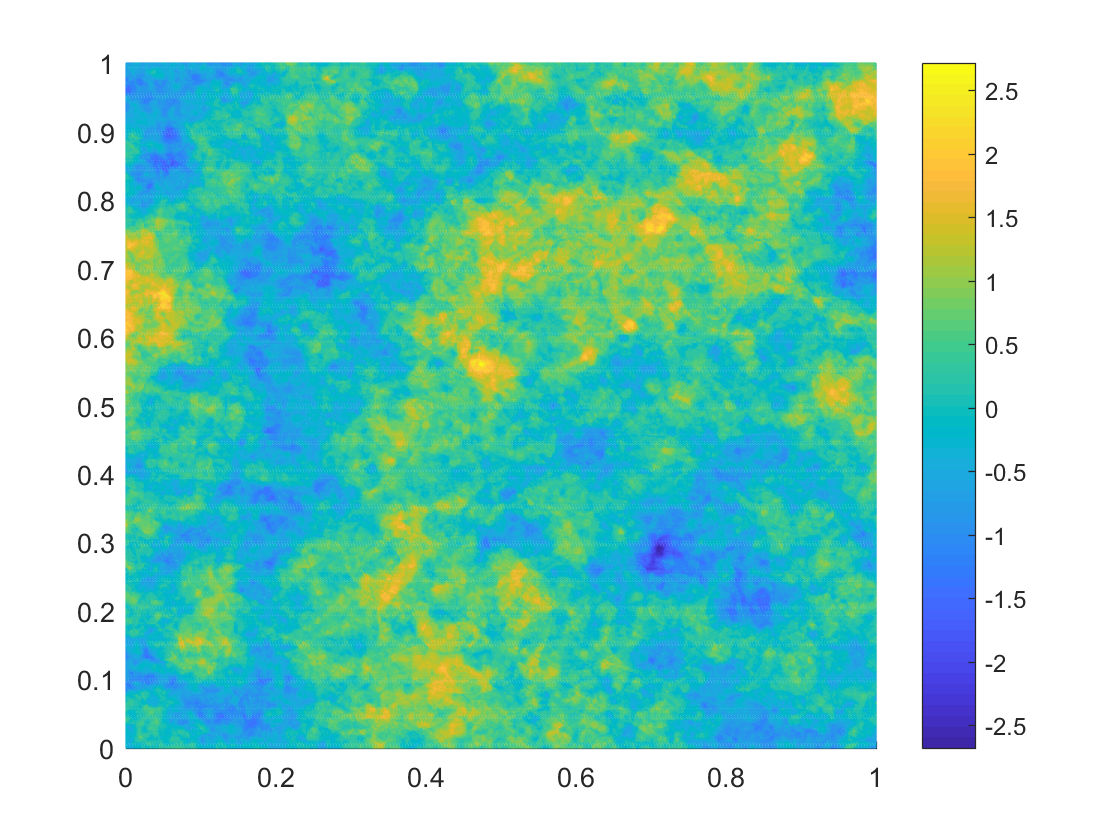}
&
\includegraphics[height=4.15cm, width=5cm]{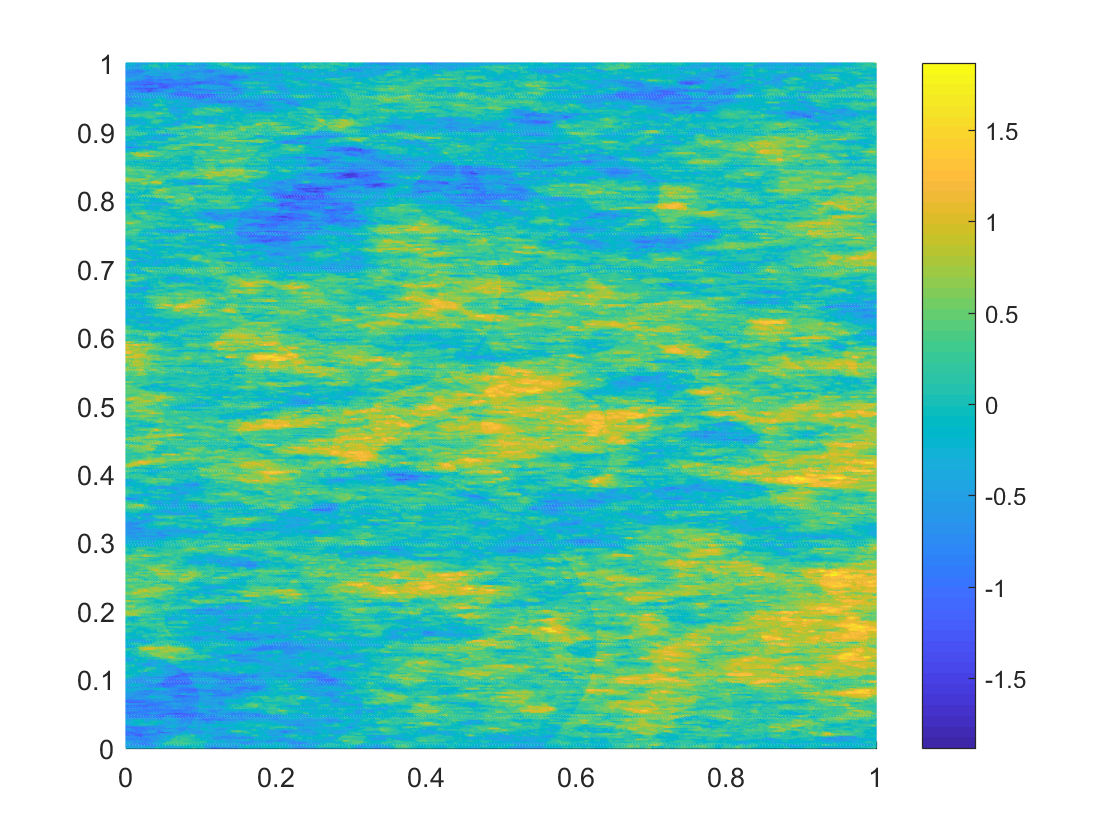}
&
\includegraphics[height=4.15cm, width=5cm]{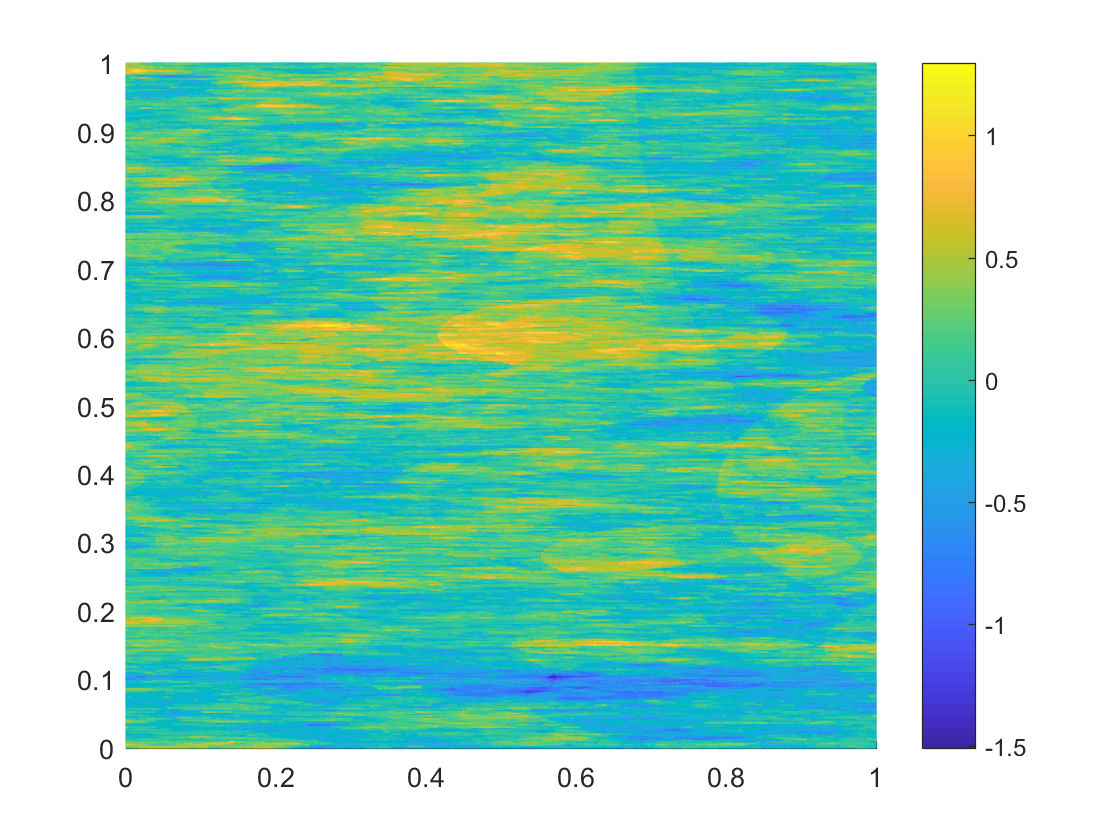}\\
\includegraphics[height=4.15cm, width=5cm]{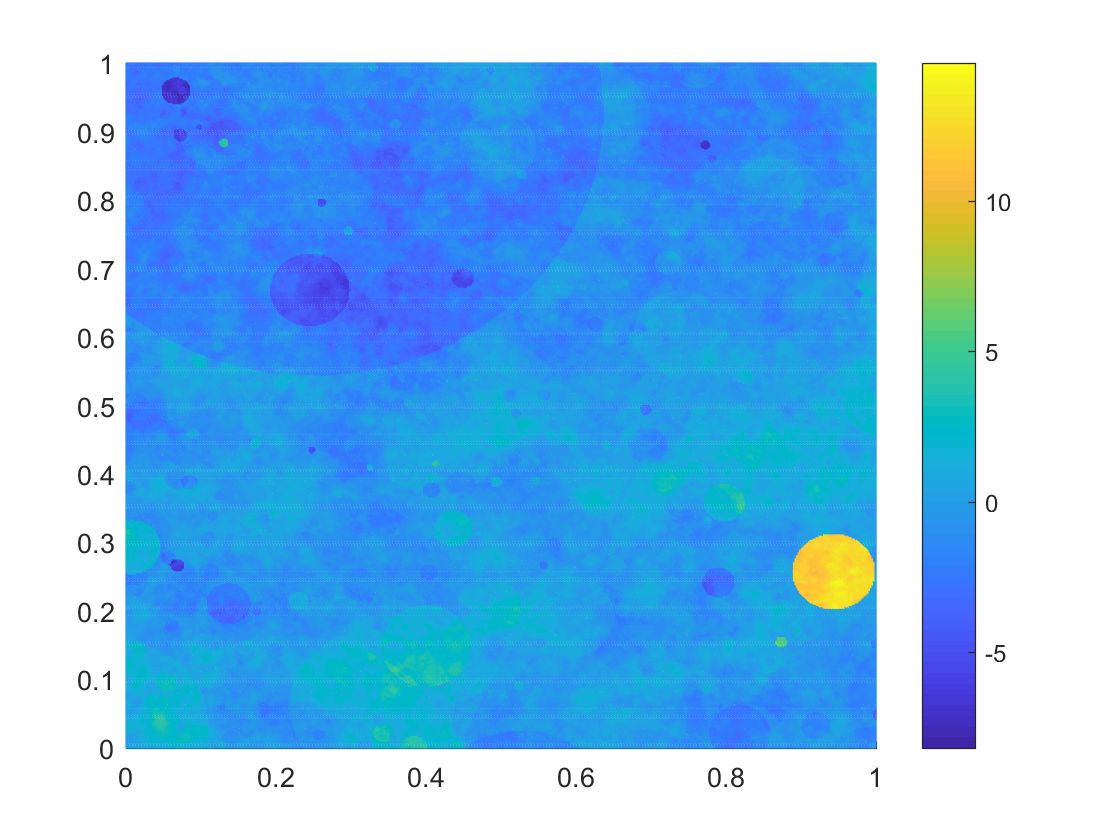}
&
\includegraphics[height=4.15cm, width=5cm]{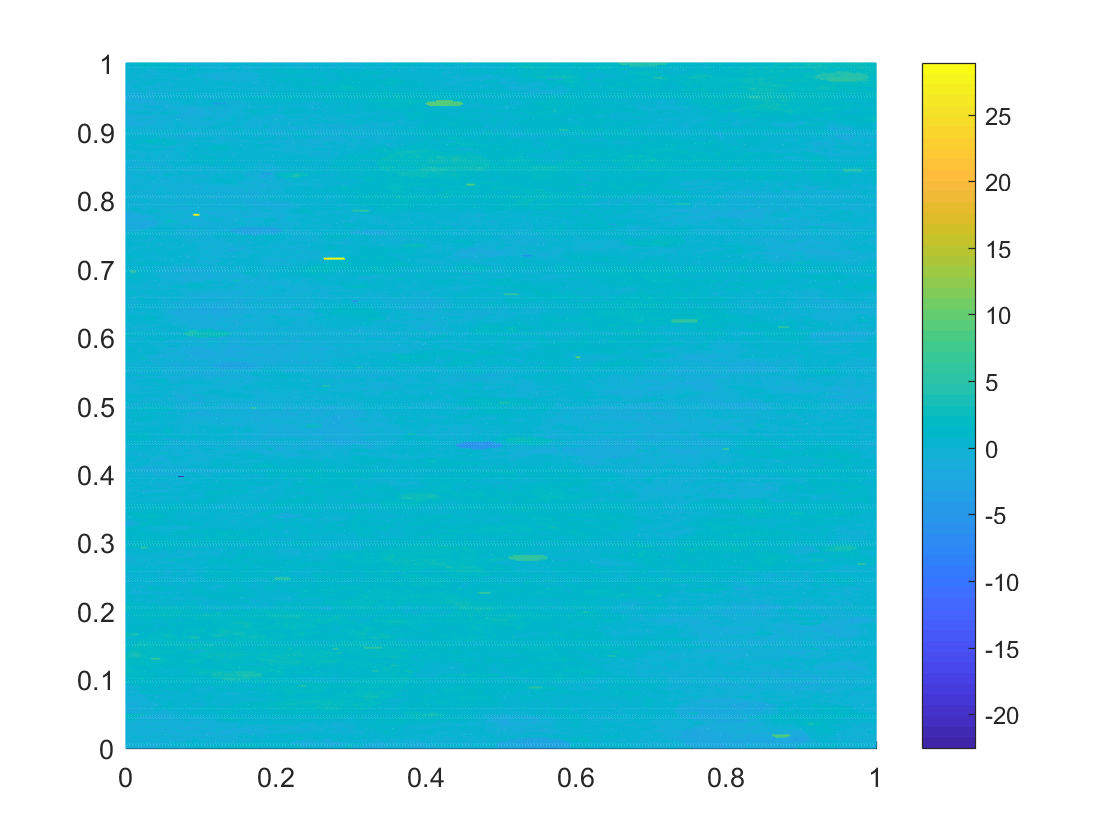}
&
\includegraphics[height=4.15cm, width=5cm]{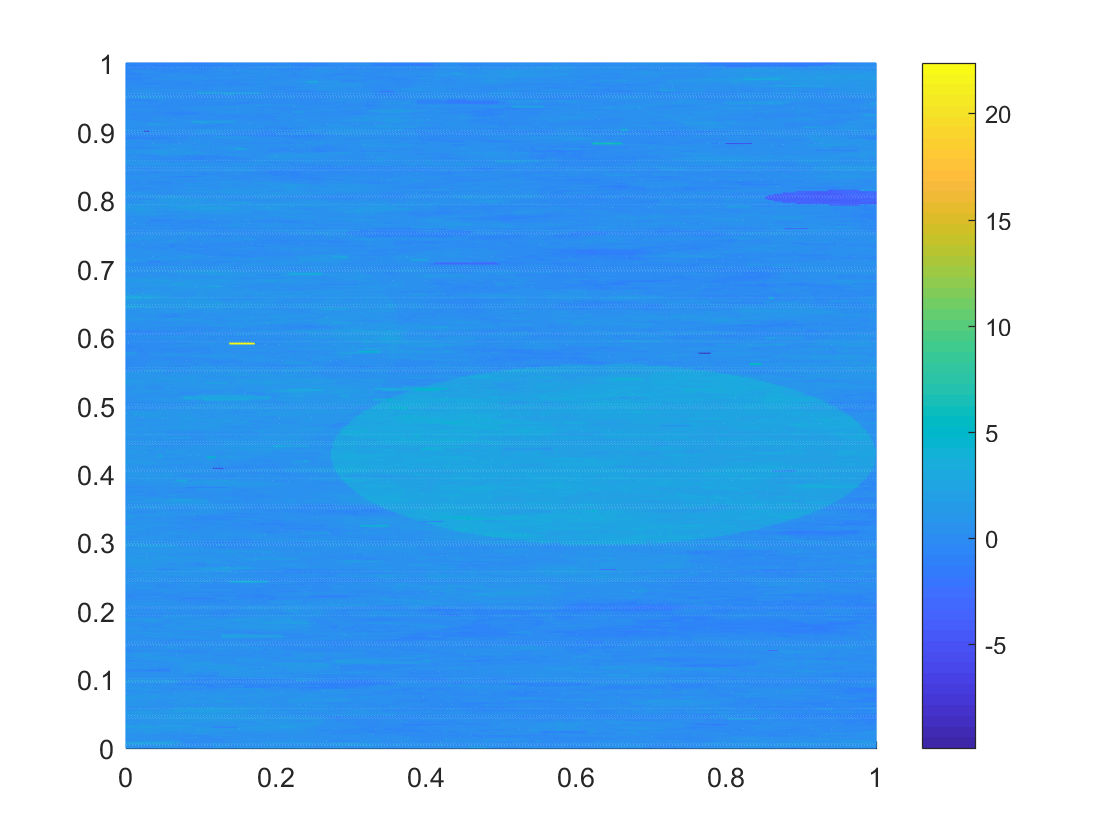}\\
$a_1=1$ (isotropic)& $a_1=1.5$&$a_1=2$
\end{tabular}
\vspace{-0.25cm}
\caption{\footnotesize{\textbf{Operator-scaling random ball} in high intensity with $H=\frac{1+a_1-\beta}{\alpha}=0.4$ and weights following a $S\alpha S(\sigma)$ distribution with $\sigma=0.1$. Top: $\alpha=2$ (Gaussian case). Bottom: $\alpha=1.8$.  } }\label{Fig2}
\end{figure}

\begin{figure}[ht!]\centering
\begin{tabular}{ccc}
\includegraphics[height=4.15cm, width=5cm]{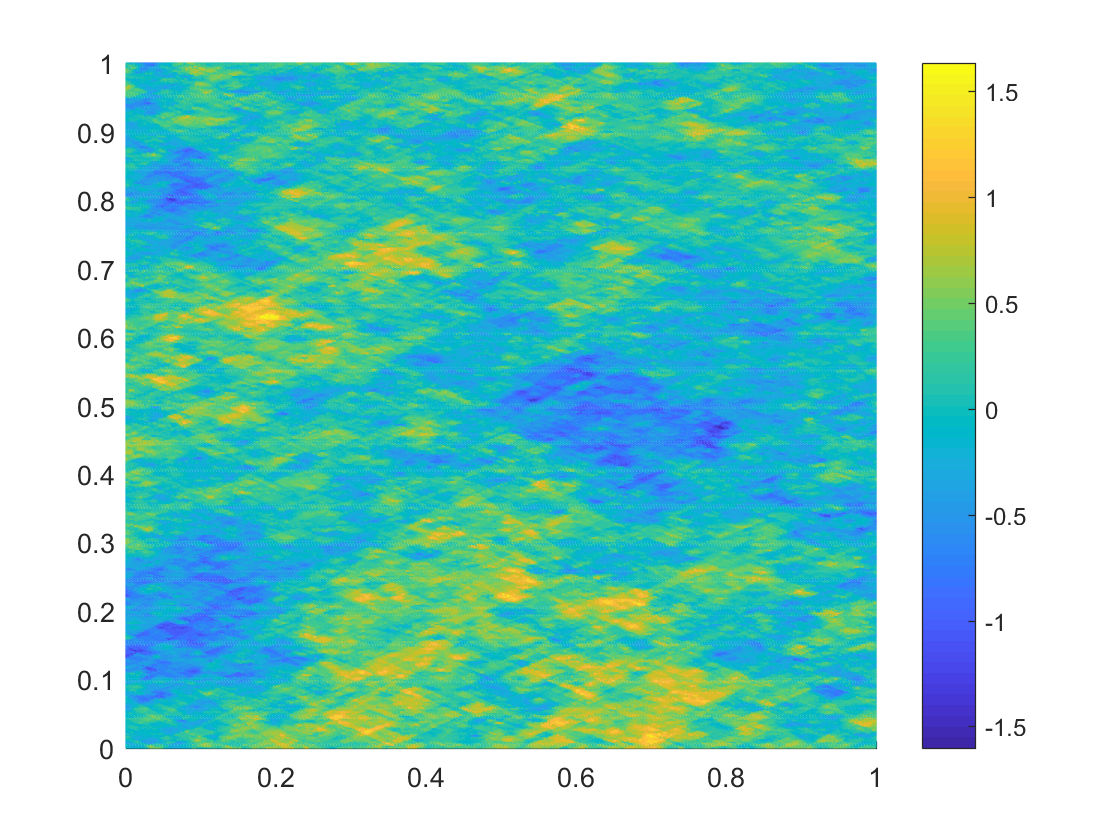}
&
\includegraphics[height=4.15cm, width=5cm]{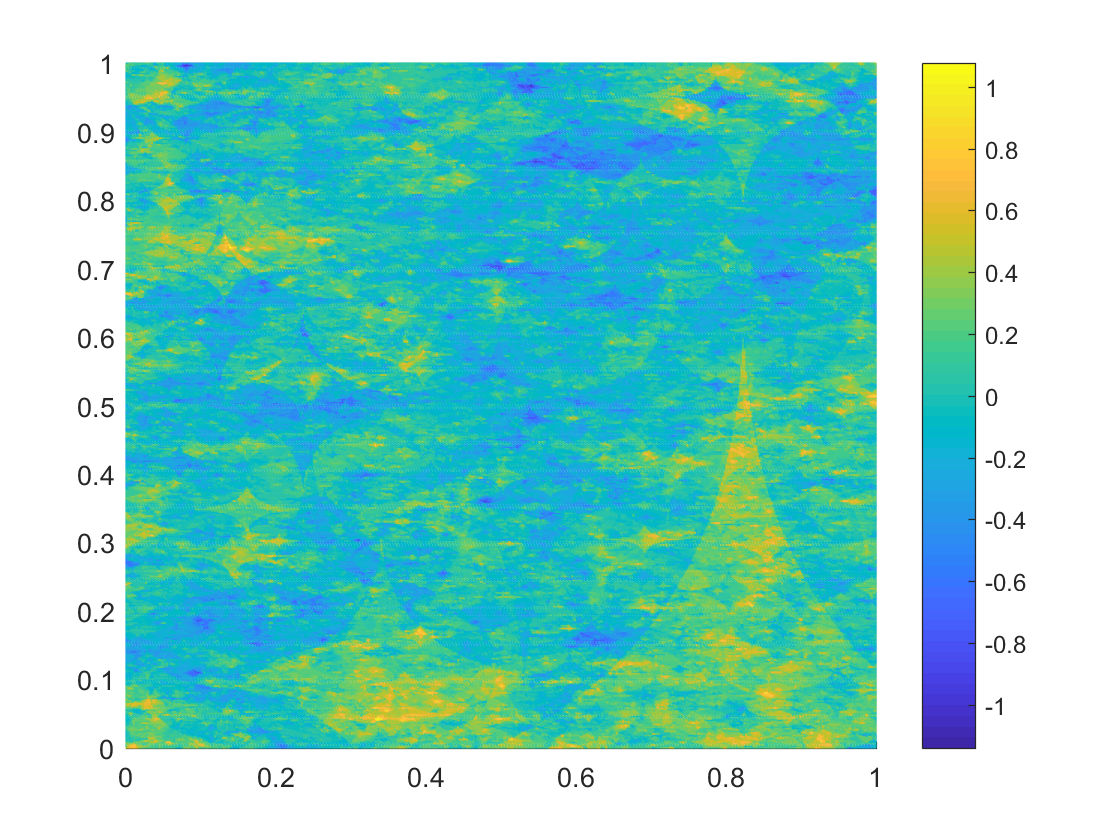}
&
\includegraphics[height=4.15cm, width=5cm]{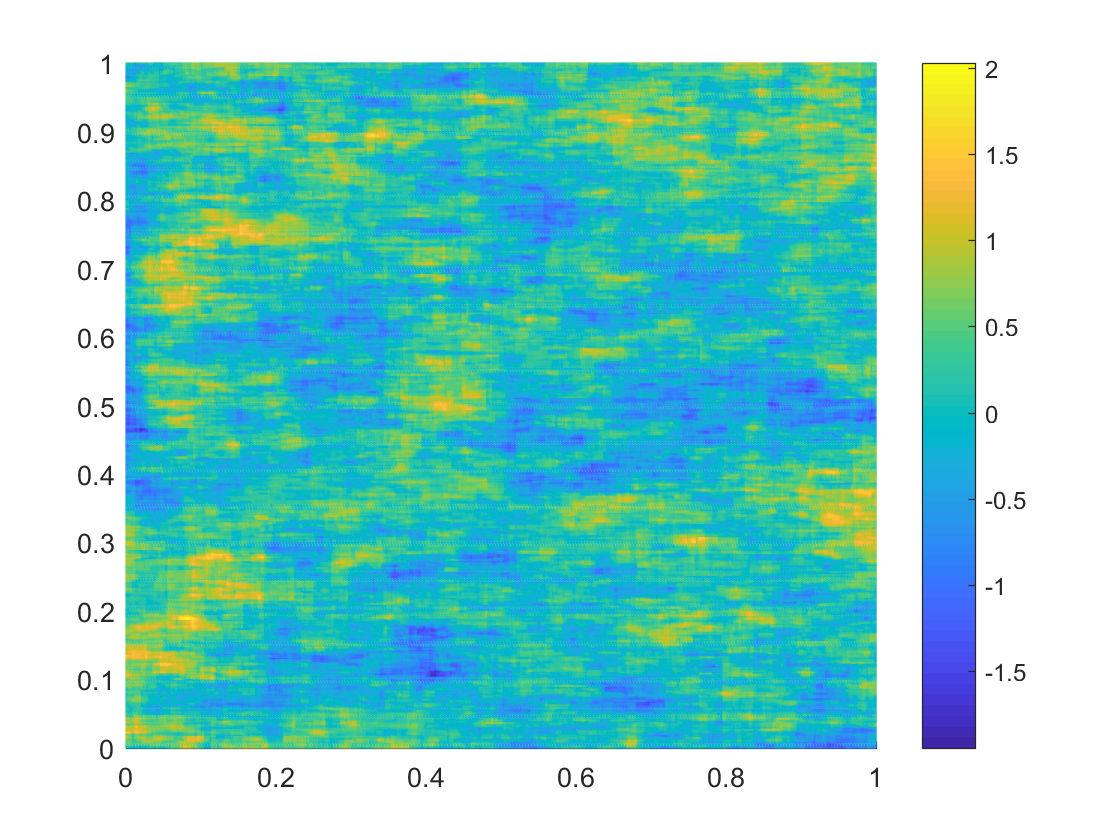}\\
\includegraphics[height=4.15cm, width=5cm]{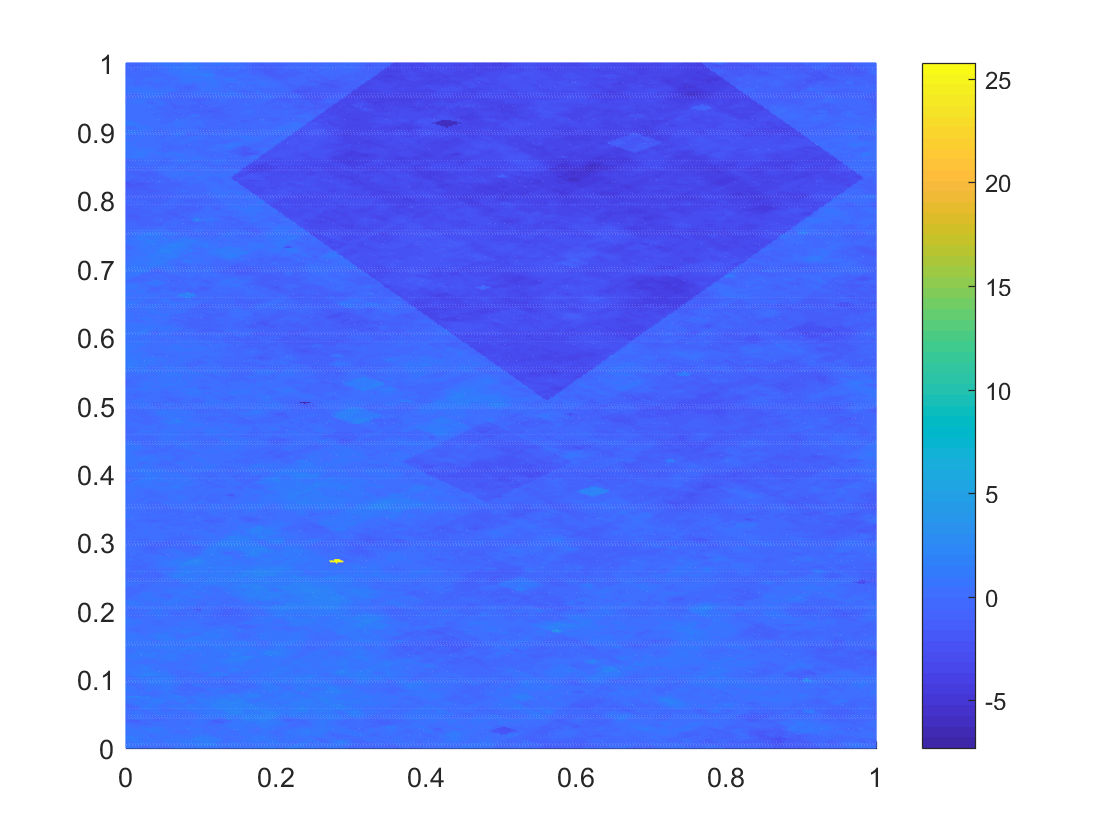}
&
\includegraphics[height=4.15cm, width=5cm]{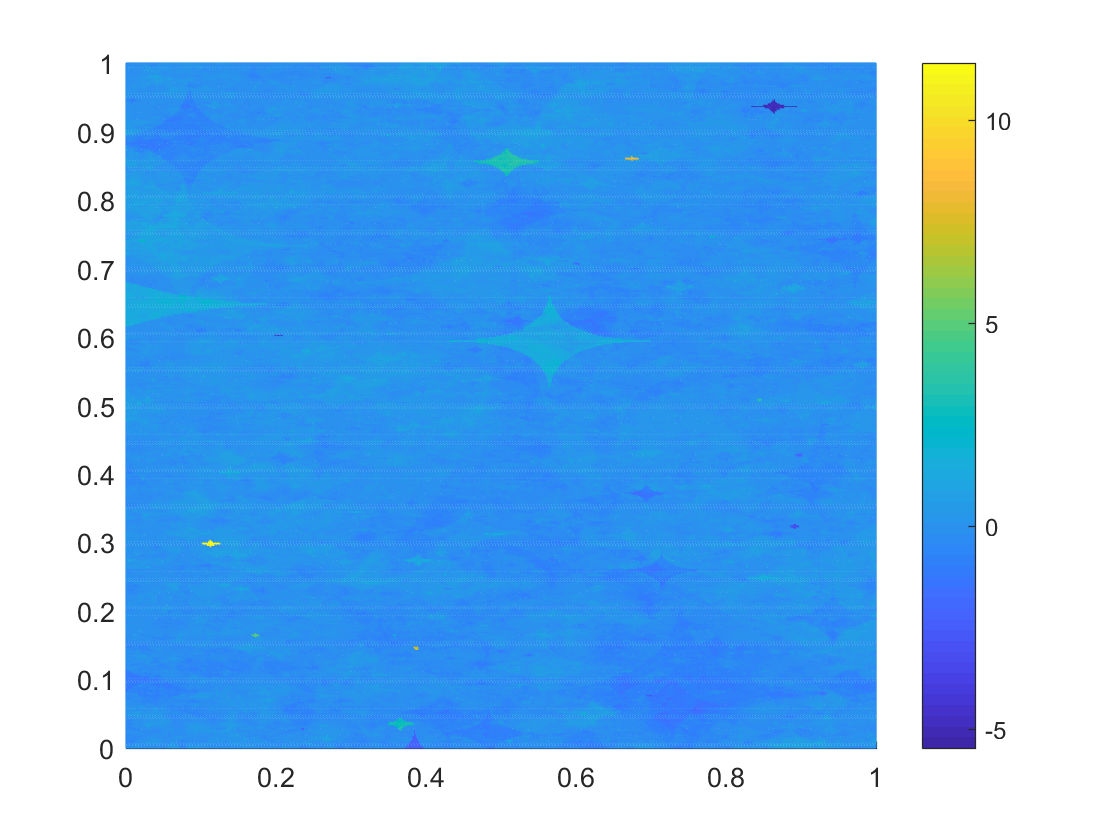}
&
\includegraphics[height=4.15cm, width=5cm]{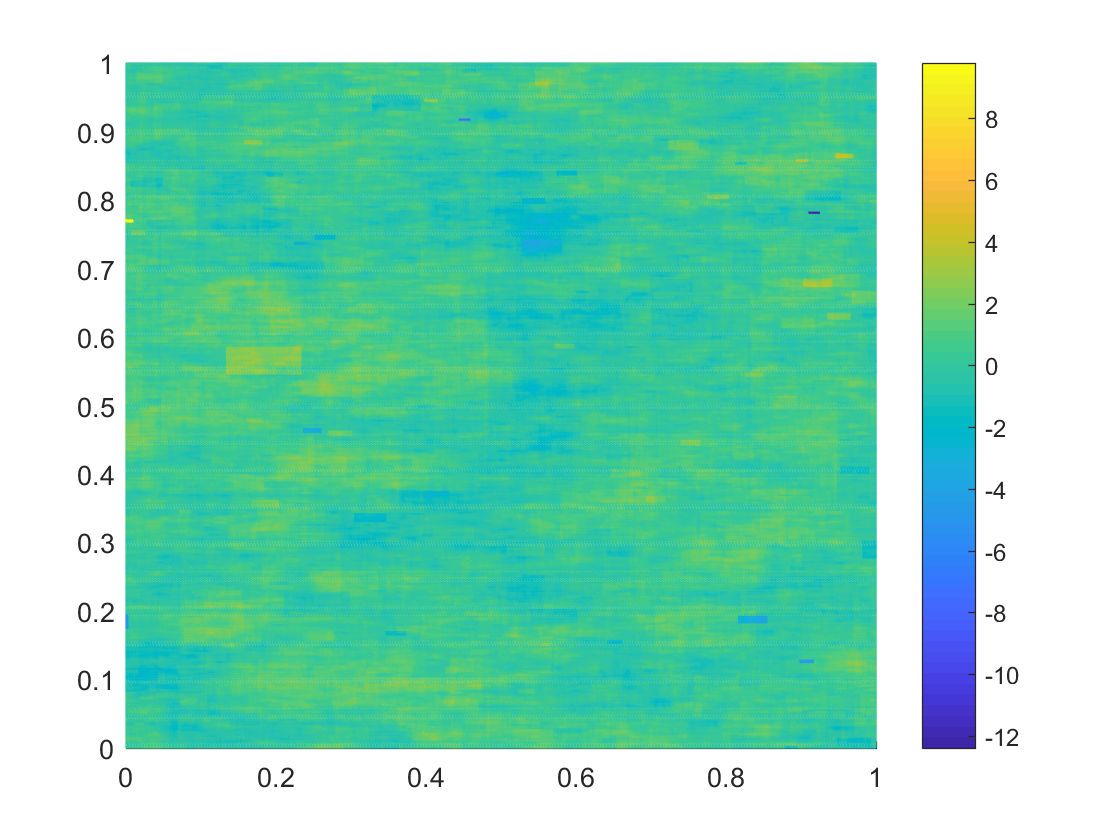}\\
$B_1$ & $B_{1/2}$&$B_\infty$
\end{tabular}
\vspace{-0.25cm}
\caption{\footnotesize{\textbf{Operator-scaling random ball} in high intensity with $H=\frac{1+a_1-\beta}{\alpha}=0.3$, $a_1=1.3$, weights following a $S\alpha S(\sigma)$ distribution with $\sigma=0.1$, and different balls: $B_1=\{x\in\R^2: |x_1|+|x_2|\le 1\}$, $B_{1/2}=\{x\in\R^2: |x_1|^{1/2}+|x_2|^{1/2}\le 1\}$ and $B_\infty=\{x\in\R^2: \max(|x_1|,|x_2|)\le 1\}$. Top: $\alpha=2$ (Gaussian case). Bottom: $\alpha=1.9$.} }\label{Fig3}
\end{figure}

\subsection*{Acknowledgments} 
YW's research was supported in part by NSA grant H98230-16-1-0322, Army Research Laboratory grant W911NF-17-1-0006, and Charles Phelps Taft Research Center at University of Cincinnati.


\bibliographystyle{apalike}
\bibliography{references}
\end{document}